\documentclass[]{article}
\usepackage{mathrsfs,subfig,latexsym,array}
\usepackage{amsmath,amssymb,amsthm,datetime,graphicx}
\usepackage{color,amsbsy}
\usepackage{caption}

\DeclareMathAlphabet{\mathpzc}{OT1}{pzc}{m}{it}

\title{
  Successive finite  element methods for Stokes equations
}\author{ Chunjae Park \thanks{Department of
    Mathematics, Konkuk University, Seoul 05029, Korea.
    \hspace{1mm}cjpark@konkuk.ac.kr}}

\begin{document}

\newtheorem{theorem}{Theorem}[section]
\newtheorem{remark}[theorem]{Remark}
\newtheorem{lemma}[theorem]{Lemma}
\newtheorem{corol}[theorem]{Corollary}
\newtheorem{proposition}[theorem]{Proposition}
\newtheorem{definition}[theorem]{Definition}
\newtheorem{assumption}{Assumption}[section]

\newcommand{\st}[2]{\mathcal{S}_{#1#2}} \newcommand{\ns}[2]{\mathcal{N}_{#1#2}}
\newcommand{\dotone}[1]{\dot{#1}}
\newcommand{\dottwo}[1]{\ddot{#1}}
\newcommand{\dotthr}[1]{\dddot{#1}}

\newcommand{\vertiii}[1]{{\left\vert\kern-0.25ex\left\vert\kern-0.25ex\left\vert #1 
    \right\vert\kern-0.25ex\right\vert\kern-0.25ex\right\vert}}

\def\disp{\displaystyle}
\def\pskip{\hspace{1pt}}
\def\mmskip{\vspace{1mm}}

\def\R{\mathbb R}
\def\O{\Omega}
\def\p{\partial}
\def\Th{\mathcal{T}_h}

\def\div{\mathrm{div}\hspace{0.5mm}}
\def\curl{\mathbf{curl}\hspace{0.5mm}}

\def\alp{\alpha}
\def\bet{\beta}
\def\gam{\gamma}
\def\del{\delta}
\def\lam{\lambda}
\def\th{\theta}

\def\f{\mathbf{f}}
\def\u{\mathbf{u}}
\def\v{\mathbf{v}}
\def\w{\mathbf{w}}
\def\x{\mathbf{x}}
\def\y{\mathbf{y}}
\def\z{\mathbf{z}}
\def\n{\mathbf{n}}
\def\btau{\boldsymbol{\tau}}
\def\bnu{\boldsymbol{\nu}}
\def\bxi{\boldsymbol{\xi}}

\def\C{\mathbf{C}}
\def\M{\mathbf{M}}
\def\G{\mathbf{G}}
\def\P{\mathbf{P}}
\def\V{\mathbf{V}}
\def\W{\mathbf{W}}
\def\X{\mathbf{X}}
\def\Y{\mathbf{Y}}

\newcommand {\snorm}[2] {| #1 |_{#2}}
\newcommand {\norm}[2] {\| #1 \|_{#2}}
\newcommand{\Evec}[2]{\overrightarrow{#1#2}}
\newcommand{\verthrnorm}[1]{{\left\vert\kern-0.25ex\left\vert\kern-0.25ex\left\vert #1 
        \right\vert\kern-0.25ex\right\vert\kern-0.25ex\right\vert}}

\def\tu{\widetilde{\u}} \def\t{\boldsymbol{\tau}}
\def\bxi{\boldsymbol{\xi}}

\def\Xh{[\mathcal{P}_h^4(\O)\cap H_0^1(\O)]^2}
\def\Mh{\mathcal{P}_h^3(\O)}
\def\ph{p_h}
\def\php{\Pi_h p}
\def\jump{\mathcal{J}}
\def\VV{\mathcal{V}}
\def\RR{\mathcal{R}}
\def\SS{\mathcal{S}}
\def\PP{\mathcal{P}}
\def\BB{\mathcal{B}}
\def\KK{\mathcal{K}}

\def\m{\mathpzc{m}}
\def\CC{\mathpzc{C}}

\def\jjp{j\hspace{1pt}j+1}
\maketitle
\begin{abstract}
  This paper  will suggest a new finite element method to find a $P^4$-velocity
  and a $P^3$-pressure solving Stokes equations.
  The method solves first the decoupled equation for the $P^4$-velocity.
  Then, four kinds of local $P^3$-pressures
  and one $P^0$-pressure will be calculated in a successive way. 
  If we superpose them, the resulting $P^3$-pressure shows the optimal order of
  convergence same as a $P^3$-projection.
  The chief time cost of the new method is on solving two linear systems for the
 $P^4$-velocity and $P^0$-pressure, respectively.
 \end{abstract}

 \section{Introduction}
 High order finite element methods for incompressible Stokes problems
 have been developed well in 2-dimensional domain and analyzed along to the inf-sup condition
 \cite{Falk2013, Guzman2018, Scott1985}.
 They, however, endure their large numbers of degrees of freedom.
 To be worse in case of Scott-Vogelius, instability appears on pressure,
 if the mesh bears singular vertices.

 Recently, we found a
 so called sting function having an interesting quadrature rule.
 In the Scott-Vogelius finite element method, it causes the discrete Stokes problem
 to be singular on the presence of exactly singular vertices.
 Even on nearly singular vertices, the pressure solution is easy to be spoiled.
 To overcome the problem based on understanding the causes,
 we did a new error analysis in a successive way 
 and restored the ruined pressure by simple post-process driven from it
 \cite{Park2020}. 

 In this paper, we will suggest a new finite element method
 to calculate a pressure solution at low cost,
 utilizing the successive way in the precedented new error analysis.
 In our method, the characteristics of a sting function
 depicted in Figure \ref{fig:sting} will also play a key role.
 
 We will solve first the decoupled equation for velocity in the $P^4$ divergence-free space
 inherited from the $\mathcal{C}^1$-Argyris $P^5$ stream function space. 
 It is simpler and smaller than the divergence-free subspace of the $P^4$-$P^3$
 Scott-Vogelius finite element space.

 The main stage of our method is the successive 5 steps calculating
 four kinds of local $P^3$-pressures for triangles, regular vertices, nearly singular vertices
 and singular corners, respectively, as well as one $P^0$-pressure in the last step. 
 They are successive in the sense that each step needs the calculated
 in the previous step.
 
 Superposition of all the calculated reaches at the final $P^3$-pressure
 which shows the optimal order of convergence
 same as a $P^3$-projection of the continuous pressure.
 The chief time cost of the new method is on solving two linear systems for the
 $P^4$-velocity and $P^0$-pressure, respectively.

The paper is organized as follows.
In the next section, the detail on finding a $P^4$-velocity  will be offered. 
After short review of nearly singular vertices in Section \ref{sec:def-sing},
we will introduce a new basis for $P^3$-pressures consisting of
non-sting, sting and constant functions and decompose the space of
$P^3$-pressures in Section \ref{sec:basis}. 
Then, we will devote Section \ref{sec:main} to describing the successive 5 steps
to find a $P^3$-pressure solution.
Finally, a numerical test will be presented in the last section.

Throughout the paper, for a set $S \subset \mathbb R^2$,
standard notations for Sobolev spaces are employed and
$L_0^2(S)$ is the space of all $f\in L^2(S)$ whose integrals over $S$ vanish.
We will use $\|\star\|_{m,S}$, $|\star|_{m,S}$ and  $(\cdot,\cdot)_S$
for the norm, seminorm for $H^m(S)$ and  $L^2(S)$  inner product, respectively. 
If $S=\O$, it may be omitted in the indices.
Denoting by $P^k$, the space of all polynomials on $\R^2$
of order less than or equal $k$,
$f\big|_S\in P^k$ will mean that $f$ coincides with a polynomial in $P^k$ on $S$.

\section{Velocity from the decoupled equation}
Let $\O$ be a simply connected polygonal domain in $\R^2$ and $\{\Th\}_{h>0}$ a regular family of triangulations of $\O$. 
Denote by $\mathcal{P}_h^k(\O)$, discrete polynomial spaces:
\[\mathcal{P}_h^k(\O) =\{ v_h \in L^2(\O)\ : \ v_h\big|_K \in P^k
  \mbox{ for all triangles } K\in\Th \}, \quad k \ge 0.\]
In this paper, we will approximate a pair of velocity and pressure
$(\u,p)\in [H_0^1(\O)]^2 \times L_0^2(\O)$ which satisfies
an incompressible Stokes problem:
\begin{equation}\label{prob:conti}
  ( \nabla \u,\nabla \v) +(p,\div \v)+(q,\div \u)
  = (\mathbf{f},\v) \quad \mbox{ for all } (\v,q) \in [H_0^1(\O)]^2\times  L_0^2(\O),
\end{equation}
for a given body force $\mathbf{f}\in [L^2(\O)]^2$.

Let $A_h^5$ be a space of $\mathcal{C}^1$-Argyris triangle elements
\cite{Brenner2002, Ciarlet} such that
\begin{equation*}\label{def:Argyris}
  A_h^5=\mathcal{P}_h^5(\O) \cap H_{0}^2(\O),
\end{equation*}
where
\[ H_{0}^2(\O)=\{ \phi\in H^2(\O)\ :\ \phi, \phi_x, \phi_y \in H_0^1(\O)\}.\]

Defining a divergence-free space $V_h$ as
\begin{equation*}
  \label{eq:def_Vh}
  V_h=\{ (\phi_{h,y},-\phi_{h,x})\  \ :\ \phi_h\in A_h^5\hspace{1pt} \} \subset
  [\mathcal{P}_h^4(\O)\cap H_0^1(\O)]^2,
\end{equation*}
we can solve  $\u_h\in V_h$ satisfying
\begin{equation}\label{eq:dclv-SC}
  (\nabla \u_h, \nabla \v_h)= (\f,\v_h)\quad \mbox{ for all } \v_h\in V_h.
\end{equation}
\begin{theorem}\label{th:vel-error}
  Let $(\u,p)\in [H_0^1(\O)]^2 \times L_0^2(\O)$ satisfy  \eqref{prob:conti}. 
  If $\u\in [H^5(\O)]^2$, then
  \begin{equation}\label{eq:th:vel-error-0}
    |\u -\u_h|_1 \le Ch^4 |\u|_5.
  \end{equation}
\end{theorem}
\begin{proof}
  If $(\u,p)\in [H_0^1(\O)]^2 \times L_0^2(\O)$ satisfies \eqref{prob:conti}, we have
  $\div\u=0$. Thus, if $\u\in [H^5(\O)]^2$,  
  there exists a stream function $\phi\in H_0^2(\O)\cap H^6(\O)$ such that \cite{GR}
  \[  \u=(\phi_y, -\phi_x).\]
  Let $\Pi_h\phi\in A_h^5$ be the projection of $\phi$ and
  $\Pi_h\u=\left((\Pi_h\phi)_y, -(\Pi_h\phi)_x\right)\in V_h$. Then we have
  \begin{equation}\label{eq:th:vel-error-1}
    |\phi-\Pi_h\phi|_2 \le C h^4 |\phi|_6,\quad |\u-\Pi_h\u|_1 \le Ch^4 |\u|_5.
  \end{equation}

  We have from \eqref{prob:conti}, \eqref{eq:dclv-SC} that
\[ (\nabla\u-\nabla\u_h,\nabla\v_h)=0\quad\mbox{ for all } \v_h\in V_h.\]
It is written in the form:
\begin{equation}\label{eq:th:vel-error-2}
  (\nabla\Pi_h\u-\nabla\u_h,\nabla\v_h)=(\nabla\Pi_h\u_h -\nabla\u,\nabla\v_h)
  \quad\mbox{ for all } \v_h\in V_h.
\end{equation}
We establish
\eqref{eq:th:vel-error-0} from \eqref{eq:th:vel-error-1}, \eqref{eq:th:vel-error-2}
with $\v_h=\Pi_h\u-\u_h\in V_h$.
\end{proof}

\section{Nearly singular vertex}\label{sec:def-sing}
\def\vts{\vartheta_{\sigma}}
A vertex $\V$ is called exactly singular if the union of all edges sharing $\V$
belongs to the union of two lines.
To be precise, let $K_1, K_2,\cdots,K_J$ be all $J$ triangles sharing $\V$
and denote by $\theta(K_k)$, the angle of $K_k$ at $\V$, $k=1,2,\cdots,J$. 
Define
\[\Upsilon(\V)=\{\theta(K_i) + \theta(K_j)\ : \ K_i\cap K_j \mbox{ is an edge, } i,j=1,2,\cdots,J \}. \]
Then $\Upsilon(\V)=\{\pi\} \mbox{ or } \emptyset$ if and only if $\V$ is exactly singular.

Since $\{\Th\}_{h>0}$ is regular, there exists $\vartheta>0$ such that
\[ \vartheta=\inf\{ \theta\ :\  \theta \mbox{ is an angle of a triangle } K\in\Th,
  h>0  \},  \]
which  depends on
the shape regularity parameter  $\sigma$ of $\{\Th\}_{h>0}$. 
Set
\begin{equation*}\label{def:vts}
  \vartheta_{\sigma} =\min(\vartheta, \pi/6),
\end{equation*}
then call a vertex $\V$ to be nearly singular if $\Upsilon(\V)=\emptyset$ or
\[  |\Theta -\pi| < \vts \mbox{ for all } \Theta\in\Upsilon(\V),   \]
otherwise regular. From the following lemma \cite{Park2020}, we note that
each interior nearly singular vertex is isolated from others.
\begin{lemma}\label{lem:isol-vtx}
  There is no interior edge connecting two nearly singular vertices.
\end{lemma}

\section{Decomposition of $\PP_h^3(\O)$}\label{sec:basis}
For each triangle $K\in\Th$, define 
\[ P^3(K)=\{ q_h\in \mathcal{P}_h^3(\O) :\ q_h=0  \mbox{ on } \O\setminus K  \}.\]
In the remaining of the paper, we will use the following notations:
\begin{itemize}
\item[] $C_\sigma$ : a generic constant which depends only the shape regularity
  parameter $\sigma$ of $\{\Th\}_{h>0}$,
\item[]  $\KK(\V)$ : the union of all triangles in $\Th$ sharing a vertex $\V$ as in Figure
  \ref{fig:union}-(b),
\item[] $\VV_h$ : the set of all vertices in $\Th$,
\item[] $|K|,|E|$: the area and length of a triangle $K$ and an edge $E$,  respectively.
\end{itemize}
  
\subsection{sting  function}
Let $\V$ be a vertex of a triangle $K$. Then there exists a unique
function $\st{\V}{K}\in P^3(K)$ satisfying the following
quadrature rule:
\begin{equation}\label{eq:quadst}
  \int_K \st{\V}{K}(x,y) q(x,y)\ dxdy = |K|q(\V)\quad \mbox{ for all } q\in
  P^3,
\end{equation}
since the both sides of \eqref{eq:quadst} are linear functionals on $P^3$.  We call
$\st{\V}{K}$ a sting function of $\V$ on $K$, named after the shape of its graph
as in Figure \ref{fig:sting}.
\begin{figure}[ht]
  \centering
  \includegraphics[width=0.52\linewidth]{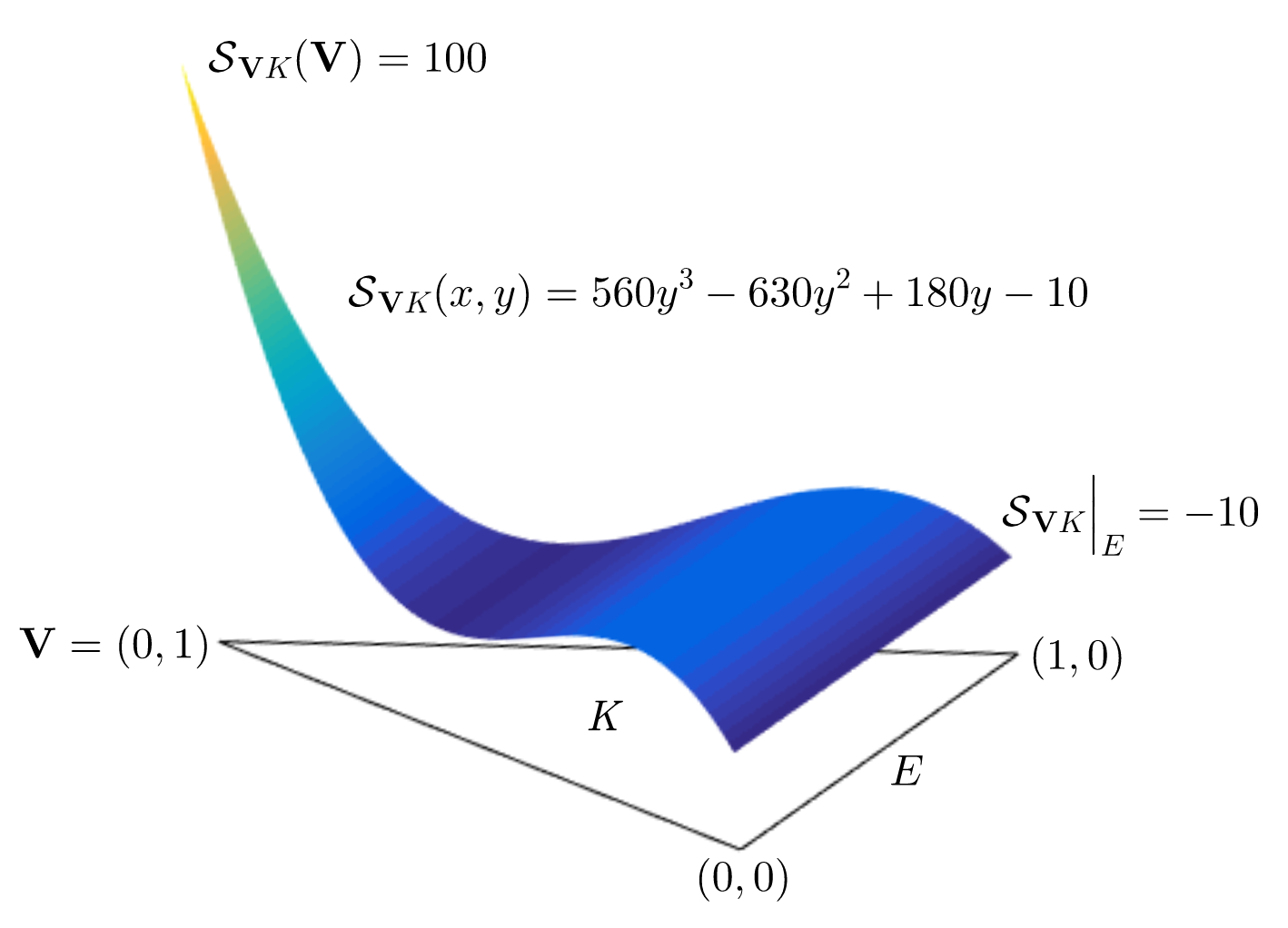}
  \caption{a sting function $\st{\V}{K}$ of  $\V$ on $K$. }
  \label{fig:sting}
\end{figure}

In the reference triangle
$\widehat{K}$ of vertices $\widehat\V_1=(0,0), \widehat\V_2=(1,0), \widehat\V_3=(0,1)$,
we have the following 3 sting functions:
\begin{equation}\label{def:sting}
  \arraycolsep=1.4pt\def\arraystretch{1.6}
  \begin{array}{lll}
    \st{\widehat{\V}_1}{\widehat{K}}&=& 560(1-x-y)^3-630(1-x-y)^2+180(1-x-y)-10,\\
    \st{\widehat{\V}_2}{\widehat{K}}&=& 560x^3-630x^2+180x-10,\\
    \st{\widehat{\V}_3}{\widehat{K}}&=& 560y^3-630y^2+180y-10. 
  \end{array}
\end{equation}
Since $\st{\V}{K}= \st{\widehat{\V}_1}{\widehat{K}} \circ F^{-1}$
for an affine transformation $F:\widehat{K}\longrightarrow K$, we estimate
\begin{equation}\label{est:stVK}
  \| \st{\V}{K} \|_{0,K} \le C_\sigma |K|^{1/2}.
\end{equation}

\subsection{non-sting function}
Given triangle $K\in\Th$ of vertices $\V_1,\V_2,\V_3$,
the following 16-point Lyness quadrature rule is exact for every
polynomial $q\in P^6$:
\begin{equation}\label{eq:Lyn-quad}
  \int_K q(x,y)\ dxdy=|K|\sum_{i=0}^{15} q(\G_i) g_i,
\end{equation}
where $\G_4, \G_5, \cdots, \G_{15}$ belong to $\p K$ and $\G_0$ is the gravity center of $K$ and
$\G_1, \G_2, \G_3$ are the centers of the medians, that is,
\[ \G_1= \frac12\V_1 + \frac14\V_2 + \frac14\V_3,\quad \G_2=
  \frac14\V_1 + \frac12\V_2 + \frac14\V_3,\quad \G_3= \frac14\V_1 +
  \frac14\V_2 + \frac12\V_3,
\]
and $g_0,g_1,\cdots,g_{15}$ are nonzero quadrature weights \cite{Lyness1975, Park2020}.

\begin{lemma}\label{lem:non-sting-unisol}
  If a cubic function $q\in P^3$ satisfies the following 10 conditions, then $q=0$.
  \begin{equation}\label{eq:lem-unisol}
    \nabla q (\G_j)=(0,0), j=0,1,2,3,\quad\mbox{ and }\quad  q(\G_0)=q(\G_1)=0.
  \end{equation}
\end{lemma} \begin{proof}
  Since $\G_0$ is the gravity center of $\G_1, \G_2, \G_3$,
  it suffices to prove $q=0$ in case
  $$\G_1=(0,0), \G_2=(1,1), \G_3=(1,-1), \G_0=(2/3,0).$$  Let
    \[a(t)=q(1,t),\ b(t)=q(2/3,t),\ c(t)=q(t,t),\ d(t)=q(t,-t).\]
  From the condition \eqref{eq:lem-unisol},
  $q$ vanishes on the line $y=0$ and
  $a'(\pm 1)=0$. It means
  \begin{equation}\label{eq:lem-nsu-1}
 a(t)=\bet (t^3-3t)\quad \mbox{ for some constant } \bet.
\end{equation}
Then since $c(0)=c'(0)=c'(1)=0,\ c(1)=a(1)=-2\bet$, we deduce
\begin{equation}\label{eq:lem-nsu-2}
  c(t)=\bet(4t^3-6t^2).
\end{equation} 
Similarly, from $d(0)=d'(0)=d'(1)=0,\ d(1)=a(-1)=2\bet$, we have
\begin{equation}\label{eq:lem-nsu-3}
  d(t)=-\bet(4t^3-6t^2).
\end{equation}
We note that $ b^{'''}=q_{yyy}=a^{'''}=6\bet$ since $q\in P^3$. Thus, we can write
\begin{equation}\label{eq:lem-nsu-4}
  b(t)=\bet t^3,
\end{equation}
from $ b(0)=b'(0)=0,\ b(2/3)=c(2/3)=-d(2/3)=-b(-2/3)$
by \eqref{eq:lem-nsu-2},\eqref {eq:lem-nsu-3}.

Now, $\bet=0$ comes from the following equation: 
  \[\bet(2/3)^3=b(2/3)=c(2/3)=\bet(4(2/3)^3-6(2/3)^2).\]
Then $q=0$ since $\bet=0$  in  \eqref{eq:lem-nsu-1}-\eqref{eq:lem-nsu-4}.
\end{proof}

From the unisolvancy in the above lemma,  for each $k=1,2,3$, we can define
two functions ${\ns{\G_k}{K}}^1, {\ns{\G_k}{K}}^2\in P^3(K)$, called non-sting,
by their following values:
\begin{equation}\label{def:nsGK}
  \arraycolsep=1.4pt\def\arraystretch{1.6}
  \begin{array}{c}
    {\ns{\G_k}{K}}^{1}(\G_0)=0,\ {\ns{\G_k}{K}}^{1}(\G_k)=0,\ 
    \nabla{\ns{\G_k}{K}}^{1}(\G_j)=(\delta_{kj},0),\quad j=0,1,2,3,\\
    {\ns{\G_k}{K}}^{2}(\G_0)=0,\ {\ns{\G_k}{K}}^{2}(\G_k)=0,\
    \nabla{\ns{\G_k}{K}}^{2}(\G_j)=(0,\delta_{kj}),\quad j=0,1,2,3,
  \end{array}
\end{equation}
where $\delta$ is the Kronecker delta.

In the reference triangle
$\widehat{K}$ of vertices $\widehat\V_1=(0,0), \widehat\V_2=(1,0), \widehat\V_3=(0,1)$,
they appear in
\begin{equation*}
  \arraycolsep=1.4pt\def\arraystretch{1.6}
  \begin{array}{lll}
    {\ns{\widehat{\G}_1}{\widehat{K}}}^1&=&\frac13(-12+75x+45y-261xy-108x^2-27y^2+228x^2y+180xy^2+40x^3-16y^3),\\
    {\ns{\widehat{\G}_1}{\widehat{K}}}^2&=&\frac13(-12+45x+75y-261xy-27x^2-108y^2+180x^2y+228xy^2-16x^3+40y^3),\\
    {\ns{\widehat{\G}_2}{\widehat{K}}}^1&=&5-26x-28y+140xy+21x^2+28y^2-112x^2y-112xy^2+8x^3,\\
    {\ns{\widehat{\G}_2}{\widehat{K}}}^2&=&\frac13(-10+57x+51y-249xy-75x^2-54y^2+228x^2y+180xy^2+16x^3+8y^3),\\
    {\ns{\widehat{\G}_3}{\widehat{K}}}^1&=&\frac13(-10+51x+57y-249xy-54x^2-75y^2+180x^2y+228xy^2+8x^3+16y^3),\\
    {\ns{\widehat{\G}_3}{\widehat{K}}}^2&=&5-28x-26y+140xy+28x^2+21y^2-112x^2y-112xy^2+8y^3.
  \end{array}
\end{equation*}

We note that
\begin{equation}\label{est:nsGK}
  \|  {\ns{\G_k}{K}}^i \|_{0,K}\le C_\sigma |K|^{1/2}|{\ns{\G_k}{K}}^i |_{1,K} \le
  C_\sigma |K|,\quad k=1,2,3,\ i=1,2.
\end{equation}
\subsection{basis for $P^3(K)$}
Given triangle $K\in\Th$ of vertices $\V_1, \V_2, \V_3$ and centers 
$\G_1, \G_2, \G_3$ of medians, define a set
\[ \BB=\{\st{\V_1}{K},\st{\V_2}{K},\st{\V_3}{K},
  {\ns{\G_1}{K}}^1, {\ns{\G_1}{K}}^2, {\ns{\G_2}{K}}^1, {\ns{\G_2}{K}}^2, {\ns{\G_3}{K}}^1,
  {\ns{\G_3}{K}}^2,1\} \subset P^3(K).\]
\begin{lemma}\label{lem:basis}
  $\BB$ is a basis for $P^3(K)$
\end{lemma} 
\begin{proof}
 For some 10 constants $\alp_1,\alp_2,\alp_3,\bet_1,\bet_2,\cdots,\bet_6,c$, assume that
  \begin{multline}\label{eq:lem:basis-0}
    q_h =\alp_1\st{\V_1}{K}+\alp_2\st{\V_2}{K}+\alp_3\st{\V_3}{K}\\
    +\bet_1{\ns{\G_1}{K}}^1+\bet_2 {\ns{\G_1}{K}}^2+\bet_3 {\ns{\G_2}{K}}^1
    +\bet_4 {\ns{\G_2}{K}}^2+ \bet_5 {\ns{\G_3}{K}}^1 +\bet_6{\ns{\G_3}{K}}^2+c =0.
  \end{multline}
  There exists a quartic function $v\in P^4$ such that 
  \begin{equation}\label{cond:test-v}
    v(\G_1)=1,\quad v(\G_2)=v(\G_3)=0,\quad v\mbox{ vanishes on } \p K.
  \end{equation}
  Then since $v_x\in P^3$ vanishes at $\V_1,\V_2,\V_3$, we expand the following
  from \eqref{def:nsGK}, \eqref{eq:lem:basis-0},\eqref{cond:test-v} and
  the quadrature rules \eqref{eq:quadst}, \eqref{eq:Lyn-quad},
  \begin{equation}\label{eq:lem:basis-spanqhw}
    0=\int_K q_h\hspace{1pt} v_x\ dxdy =\int_K \check{q}_h \hspace{1pt}v_x\ dxdy
    =-\int_k \check{q}_{h,x} \hspace{1pt}v\ dxdy
    =-|K|\bet_1 {\ns{G_1}{K}}_x^1(\G_1) v(\G_1)g_1,
  \end{equation}
  where $\check{q}_h=q_h-\alp_1\st{\V_1}{K}-\alp_2\st{\V_2}{K}-\alp_3\st{\V_3}{K}.$
  Thus, $\bet_1=0$ and similarly $\beta_j=0,j=2,3,\cdots,J$.
  
  There also exists a quartic function $w\in P^4$ such that
  \begin{equation}\label{cond:test-w}
    \nabla w(\V_1)\cdot\Evec{\V_1}{\V_2}=1,\ 
    \nabla w(\V_2)\cdot\Evec{\V_1}{\V_2}=0,\  \int_{\overline{\V_1\V_2}} w\ d\ell=0,\ 
    w\mbox{ vanishes on } \overline{\V_1\V_3}\cup  \overline{\V_2\V_3}.
  \end{equation}
  Then with $z=\nabla w \cdot\Evec{\V_1}{\V_2}\in P^3$,
  we get the following from \eqref{cond:test-w} and
  the quadrature rule  \eqref{eq:quadst},
  \begin{equation*}
    0=\int_K \left(\alp_1\st{\V_1}{K}+\alp_2\st{\V_2}{K}+\alp_3\st{\V_3}{K}+c\right)z\ dxdy
    =\alp_1|K| z(\V_1).
  \end{equation*}
It means $\alp_1=0$ and similarly $\alp_2=\alp_3=0$.  
\end{proof}

\subsection{decomposition of $\PP_h^3(\O)$}
Given triangle $K$, let $\ns{h}{}(K)$
be a space spanned by 6 non-sting functions in $P^3(K)$, that is,
\[ \ns{h}{}(K)=  <{\ns{\G_1}{K}}^1,{\ns{\G_1}{K}}^2, {\ns{\G_2}{K}}^1, {\ns{\G_2}{K}}^2, {\ns{\G_3}{K}}^1, {\ns{\G_3}{K}}^2>.  \]
Then by Lemma \ref{lem:basis}, we can decompose $P^3(K)$ into
\begin{equation}\label{eq:decom-pK}
P^3(K)=\ns{h}{}(K) \bigoplus <\st{\V_1}{K},\st{\V_2}{K},\st{\V_3}{K},1>.
\end{equation}
Given vertex $\V$, let $\st{h}{}(\V)$ be a space spanned by all sting
functions of $\V$, that is,
\[ \st{h}{}(\V)= <\st{\V}{K_1},\st{\V}{K_2},\cdots, \st{\V}{K_J}>,\]
where $K_1, K_2, \cdots, K_J$ are all triangles in $\Th$ sharing $\V$.

We call $q_h^\V\in\st{h}{}(\V)$ a sting pressure of $\V$ and  $q_h^K\in \ns{h}{}(K)$  a non-sting pressure of $K$. Examples of their supports are depicted in Figure \ref{fig:qhvqhK}.

\begin{figure}[hb]
  \centering
  \subfloat[The support of $q_h^\V$, a sting pressure of $\V$ \newline
 $q_h^\V$ is determined by $q_h^\V\big|_{K_j}(\V),\ j=1,2,3,4,5.$]{
    \includegraphics[width=0.38\textwidth]{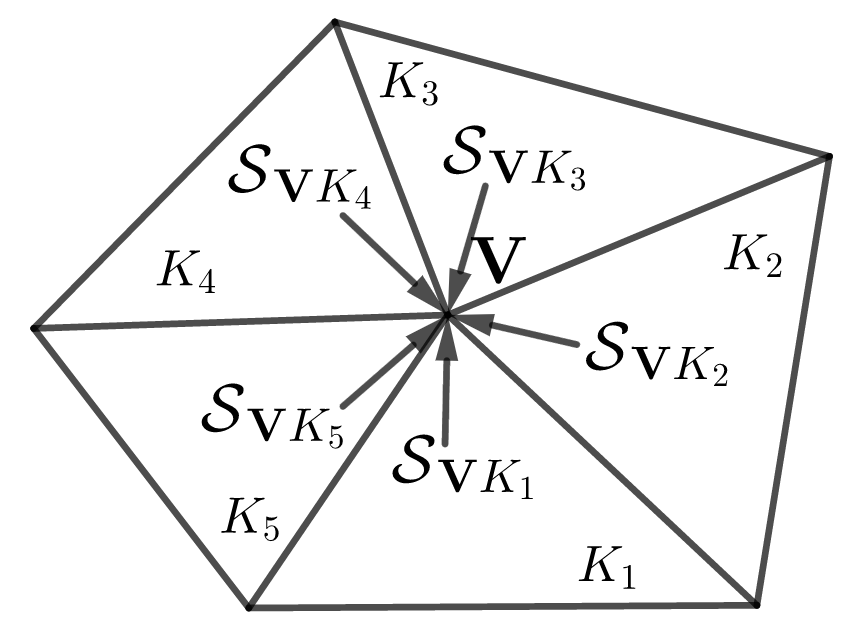}}\qquad\qquad
  \subfloat[The support of $q_h^K$, a non-sting pressure of $K$\newline
  $q_h^K$ is determined by $\nabla q_h^K(\G_j),\ j=1,2,3.$
  ]{\raisebox{3ex}{
    \includegraphics[width=0.38\textwidth]{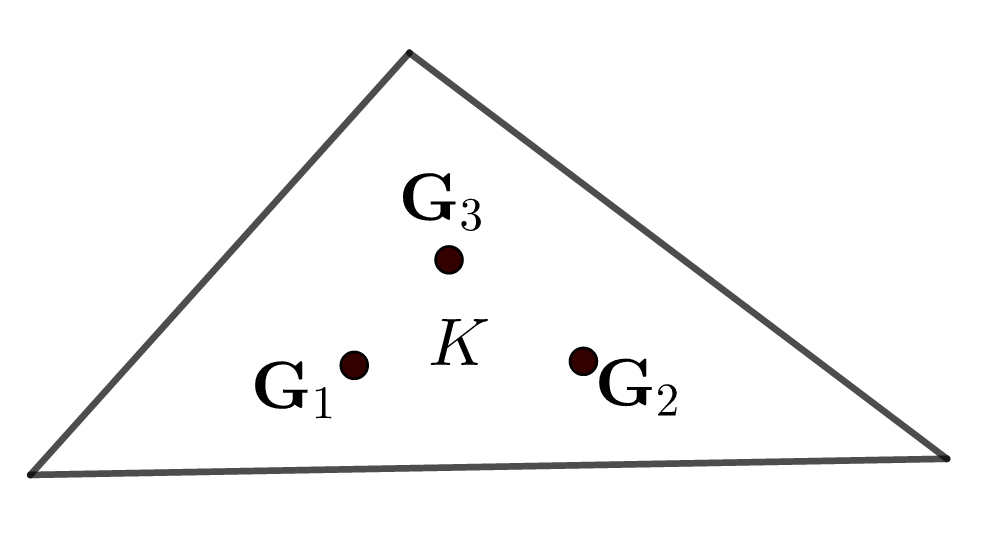}}}
  \caption{
    characteristics of $q_h^\V\in \st{h}{}(\V)$ and $q_h^K\in\ns{h}{}(K)$}
  \label{fig:qhvqhK}
\end{figure}

Then from \eqref{eq:decom-pK}, we can decompose $\PP_h^3(\O)$ into
\begin{equation}\label{eq:decom-ph3}
  \Mh=\left(\bigoplus_{K\in\Th}\ns{h}{}(K)\right) \bigoplus \left(\bigoplus_{\V\in\VV_h}\st{h}{}(\V)\right)
  \bigoplus \mathcal{P}_h^0(\O).
\end{equation}

\section{Successive method for pressure}\label{sec:main}
In this section, we will find the following discrete pressures:
\begin{equation*}
\ph^K \in \ns{h}{}(K), \quad \ph^{\V}\in \st{h}{}(\V),\quad \ph^{\CC}\in \mathcal{P}_h^0(\O),
\end{equation*}
in a successive way consisting of the following 5 steps.

\begin{description}
  \item[Step 1.] 
  Calculate a non-sting pressure $\ph^K\in \ns{h}{}(K)$ using $\u_h, \f$ for each triangle $K\in \Th$ and superpose  them to form 
  \[ \ph^{\ns{}{}} =\sum  \ph^K.   \]
\item[Step 2.]  Calculate a sting pressure $\ph^{\V_r}\in \st{h}{}(\V_r)$  using $\ph^{\ns{}{}}, \u_h, \f$
  for each regular vertex $\V_r$ and superpose them to form
  \[ \dotone{\ph} = \ph^{\ns{}{}} + \sum \ph^{\V_r}.\] 
\item[Step 3.]
  Calculate a sting pressure $\ph^{\V_s}\in \st{h}{}(\V_s)$ using  $\dotone{\ph}, \u_h, \f$ for each nearly singular vertex $\V_s$ except corners and superpose them to form
\[ \dottwo{\ph} = \dotone{\ph} + \sum \ph^{\V_s}.\]
\item[Step 4.]
Calculate a sting pressure $\ph^{\V_c}\in \st{h}{}(\V_c)$ using $\dottwo{\ph}, \u_h, \f$
  for each nearly singular corner $\V_c$  and superpose them to form
\[ \dotthr{\ph} = \dottwo{\ph} + \sum \ph^{\V_c}.\]
\item[Step 5.]
  Calculate a piecewise constant pressure $\ph^{\CC}\in \mathcal{P}_h^0(\O)$
  using $\dotthr\ph, \u_h, \f$.
\end{description}
The detail of Step $i$ above
will be given in Section \ref{sec:main}.$i$ below, $i=1,2,3,4,5$.

If we define $\ph\in \mathcal{P}_h^3(\O)$ as
\begin{equation}\label{def:ph}
  \ph = \dotthr\ph+\ph^{\CC},
\end{equation}  
it shows the optimal order of convergence as in the following theorem.
\begin{theorem}\label{th:prs-error}
  Let $(\u,p)\in [H_0^1(\O)]^2 \times L_0^2(\O)$ satisfy
  \eqref{prob:conti}.
  If $(\u,p)\in  [H^5(\O)]^2\times H^4(\O)$, then
  \begin{equation}\label{eq:th:prs-error}
    \|p-\ph\|_0 \le C h^4 (|\u|_5 + |p|_4).
  \end{equation}
\end{theorem}

\setcounter{subsection}{-1}
\subsection{decomposition of $\php$}
Let $\php\in \Mh$ be a Hermite interpolation of $p$ in the theorem \ref{th:prs-error}
such that
\begin{equation}\label{def:php}
  \nabla\php (\V)=\nabla p(\V),\quad  \php (\V)=p(\V),\quad \php(\G)=p(\G),
\end{equation}
at all vertices $\V$ and gravity centers $\G$ of triangles in $\Th$.
Then from \eqref{eq:decom-ph3}, we can split $\php\in\Mh$ into
\begin{equation}\label{eq:decom-pihp}
\Pi_h p= \sum_{K\in\Th}(\php)^K + \sum_{\V\in \VV_h} (\php)^{\V} +  (\php)^{\CC},
\end{equation}
where 
\begin{equation}
(\php)^K \in \ns{h}{}(K), \quad (\php)^{\V}\in \st{h}{}(\V),\quad (\php)^{\CC}\in \mathcal{P}_h^0(\O).
\end{equation}
We note that $\php$ satisfies that 
\begin{equation}\label{eq:pihp0}
  (\Pi_hp,\div \v)=(\f,\v)-(\nabla\u,\nabla\v)-(p-\Pi_hp,\div\v)\quad \mbox{ for all }
  \v\in \left[H_0^1(\O)\right]^2.
\end{equation}

We assume the following to exclude pathological meshes.
\begin{assumption}\label{asm:Th}
$\Th$ has no two adjacent triangles whose union has two corner points  of $\p\O$.
\end{assumption}

\subsection{ Step 1. non-sting pressure for triangle}\label{sec:non-sting}
Let's fix a triangle $K$ of vertices $\V_1, \V_2, \V_3$ and
centers $\G_1,\G_2,\G_3$ of medians.
There exists a unique
  function $v\in \mathcal{P}_h^4(\O)\cap H_0^1(\O)$ such that
  \[ v(\G_1)=1,\quad v(\G_2)=v(\G_3)=0,\quad v \mbox{ vanishes on }
    \O\setminus K. \]
  Then by same argument as in \eqref{eq:lem:basis-spanqhw} with a test function
  $\v_h=(v,0)$, all 6 non-sting basis functions of $\ns{h}{}(K)$ satisfy
  \begin{equation}\label{eq:ns-G1w}
    ({\ns{G_k}{K}}^i,\div{\v_h})=
    \left\{\arraycolsep=1.4pt\def\arraystretch{1.6}
        \begin{array}{l}
          -|K| {\ns{G_1}{K}}_x^1(\G_1) v(\G_1)g_1 = -|K|g_1,\ \mbox{ if } k=1 \mbox{ and } i=1,\\
          0,\quad \mbox{ otherwise}.
        \end{array}\right.
      \end{equation}

Generally, for each $k=1,2,3$, choose $v_k\in \mathcal{P}_h^4(\O)\cap H_0^1(\O)$ such that
\begin{equation}\label{eq:lem-ns-error-defwk}
  v_k(\G_j)=\delta_{kj}, j=1,2,3,\quad v_k \mbox{ vanishes on } \O\setminus K, 
\end{equation}
and define 
\begin{equation}\label{def:whk}
  {\v_{h,k}}^1=(v_k,0),\quad  {\v_{h,k}}^2=(0,v_k).
\end{equation}
Then by same argument as in \eqref{eq:ns-G1w} with \eqref{def:nsGK}, 
 \eqref{eq:lem-ns-error-defwk}, \eqref{def:whk}, all 6 basis functions of $\ns{h}{}(K)$ satisfy
\begin{equation}\label{eq:ns-Gkw}
 ({\ns{G_k}{K}}^i,\div{\v_{h,l}}^j)=-|K| g_k \delta_{kl}\delta_{ij},\quad k,l=1,2,3,\ i,j=1,2.
\end{equation}

Now with the aid of \eqref{eq:ns-Gkw},
we can find a unique $\ph^K\in \ns{h}{}(K)$ such that
\begin{equation}\label{eq:q_G}
  (p_h^K, \div \v_h) = (\f,\v_h)-(\nabla\u_h, \nabla\v_h)
  \quad \mbox{ for alll } \v_h\in \{{\v_{h,k}}^i\ :\ k=1,2,3,\ i=1,2\}.
\end{equation}
Actually, if we set $\ph^K$
for 6 unknown constants $\alp_{11}, \alp_{12}, \alp_{21}, \alp_{22}, \alp_{31}, \alp_{32}$ as
  \[ p_h^K= \alp_{11}{\ns{\G_1}{K}}^1+\alp_{12}{\ns{\G_1}{K}}^2
    +\alp_{21} {\ns{\G_2}{K}}^1 +\alp_{22}{\ns{\G_2}{K}}^2
    +\alp_{31}{\ns{\G_3}{K}}^1+\alp_{32} {\ns{\G_3}{K}}^2, \]
  we deduce that
  \begin{equation}\label{eq:lem-ns-error-sol}
    \alp_{k,i}=-|K|^{-1}g_k^{-1}\left((\f,{\v_{h,k}}^i)-(\nabla\u_h, \nabla{\v_{h,k}}^i)\right),\quad
    k=1,2,3,\ i=1,2.
  \end{equation}
  
\begin{lemma}\label{lem:ns-error}
  \begin{equation}\label{eq:ns-error}
    \|(\Pi_hp)^K-p_h^K\|_{0,K} \le C_\sigma  (|\u-\u_h|_{1,K}+\|p-\php\|_{0,K}).
  \end{equation}
\end{lemma}
\begin{proof}
  Let $V=\{{\v_{h,k}}^i\ :\ k=1,2,3,\ i=1,2\}$.  We can rewrite \eqref{eq:pihp0} for
  $\v_h\in V$ into
  \begin{equation}\label{eq:lem:ns-error-php}
    \left((\Pi_hp)^K, \div \v_h\right) = (\f,\v_h)-(\nabla\u,
    \nabla\v_h)-(p-\Pi_hp,\div\v_h), 
  \end{equation}
  since $(1,\div\v_h)=0$ from \eqref{eq:lem-ns-error-defwk}, \eqref{def:whk}
  and by quadrature rule in \eqref{eq:quadst},
  $$(\st{\V_k}{K},\div\v_h)=|K|\div\v_h(\V_k)=0,\ k=1,2,3.$$
  
If we denote the error by
  $e_h^K=(\Pi_hp)^K-p_h^K$, then from \eqref{eq:q_G}, \eqref{eq:lem:ns-error-php}, it satisfies
  \begin{equation}\label{eq:lem:ns-error-ehk} (e_h^K, \div \v_h) = -(\nabla\u-\nabla\u_h,
    \nabla\v_h)-(p-\Pi_hp,\div\v_h)\quad\mbox{ for all } \v_h\in V.
  \end{equation}
Represent $e_h^K$ with 6 constants $e_{11}, e_{12}, e_{21}, e_{22}, e_{31}, e_{32}$ as
\begin{equation}
  \label{eq:lem-ns-error-spaneh}
  e_h^K= e_{11}{\ns{\G_1}{K}}^1+e_{12}{\ns{\G_1}{K}}^2
    +e_{21} {\ns{\G_2}{K}}^1 +e_{22}{\ns{\G_2}{K}}^2+e_{31}{\ns{\G_3}{K}}^1+e_{32} {\ns{\G_3}{K}}^2.
  \end{equation}
  Then by same argument inducing \eqref{eq:lem-ns-error-sol}, we have
  \begin{equation}\label{eq:lem-ns-error-ehk2}
    e_{ki} = |K|^{-1}g_k^{-1}\left((\nabla\u-\nabla\u_h,\nabla{\v_{h,k}}^i)
      +(p-\php, \div{\v_{h,k}}^i)\right),\quad
    k=1,2,3,\ i=1,2.
  \end{equation}
  We note that  $|v_k|_1 \le C_\sigma,\ k=1,2,3$ from \eqref{eq:lem-ns-error-defwk}.
  Thus, \eqref{eq:ns-error} comes from \eqref{est:nsGK}, \eqref{eq:lem-ns-error-spaneh}, \eqref{eq:lem-ns-error-ehk2}.
\end{proof}
Superpose all the calculated non-sting pressures $\ph^K, K\in\Th$ and define
\begin{equation}\label{def:phns}
  \ph^{\ns{}{}}=\sum_{K\in\Th} \ph^K,\quad (\php)^{\ns{}{}}= \sum_{K\in\Th}(\php)^K.
\end{equation}

\subsection{Step 2. sting pressure for regular vertex}\label{sec:sting-regular}
\subsubsection{test functions in two  triangles}
Let $\V$ be a vertex and $K_1, K_2$ two back-to-back triangles sharing $\V$
as in Figure \ref{fig:union}-(a). Denote by $\W_0,\W_1,\W_2, \t$, other 3 vertices and a unit
tangent vector such that
\[ \overline{\V\W_0}= K_1\cap K_2,\quad \W_j\in K_j\setminus\{\V,\W_0\}, j=1,2,\quad
 \t = \frac{\Evec{\V}{\W_0}}{|\overline{\V\W_0}|}. \]

There exists a function $w\in \mathcal{P}_h^4(\O)\cap H_0^1(\O)$
such that
\begin{equation}\label{cond:w}
  \frac{\p w}{\p\t}(\V)=1,\ \frac{\p w}{\p\t} (\W_0)=0,\ \int_{\overline{\V\W_0}} w\ d\ell=0,\
   \mbox{ the support of  } w \mbox{ is } K_1\cup K_2.
\end{equation}
Assume $K_1,K_2$ are counterclockwisely numbered with respect to $\V$.
Then by simple calculation, we have
  \begin{equation}
    \nabla w\big|_{K_1}(\V) =\frac{|\overline{\V\W_0}|}{2|K_1|}\
    {\Evec{\V}{\W_1}}^{\perp},\quad
    \nabla w\big|_{K_2}(\V) =-\frac{|\overline{\V\W_0}|}{2|K_2|}\
    {\Evec{\V}{\W_2}}^{\perp},
  \end{equation}
  where $(\star)^{\perp}$ denotes the $90^\circ$ counterclockwise
  rotation of vector $\star$.
  Thus if we set a test function $\w_h^{\bxi}=w\bxi=(\xi_1w, \xi_2w)$ for a vector
  $\bxi=(\xi_1,\xi_2)$, it satisfies 
  \begin{equation}\label{eq:divwhval}
    \div\w_h^{\bxi}\big|_{K_1} (\V)=\frac{|\overline{\V\W_0}|}{2|K_1|}\ {\Evec{\V}{\W_1}}^{\perp}\cdot \bxi,\quad
    \div\w_h^{\bxi}\big|_{K_2} (\V)=-\frac{|\overline{\V\W_0}|}{2|K_2|}\ {\Evec{\V}{\W_2}}^{\perp}\cdot \bxi,
  \end{equation}
  and $\div\w_h^{\bxi}$ vanishes at all other vertices.
  
  Let $q_h\in\Mh$ be represented with some constants $\alp_j, \bet_j,j=1,2,3$ as
  \[ q_h =\alp_1\st{\V}{K_1} + \alp_2\st{\W_0}{K_1}+
    \alp_3\st{\W_1}{K_1} +\bet_1\st{\V}{K_2} +
    \bet_2\st{\W_0}{K_2}+ \bet_3\st{\W_2}{K_2},\]
  that is, it is a sum of sting pressures on $K_1\cup K_2$.
  Then, by quadrature
  rule of sting functions in \eqref{eq:quadst}, we expand from \eqref{eq:divwhval} that
  \begin{equation*}
    \arraycolsep=1.0pt\def\arraystretch{2.5}
    \begin{array}{lll}
    (q_h^\V, \div\w_h^{\t})
    &=&|K_1|\div\w_h^{\t}\big|_{K_1}(\V)+|K_2|\div\w_h^{\t}\big|_{K_2}(\V)
    = \disp\frac{1}2\left(\alp_1{\Evec{\V}{\W_1}}^{\perp}
      -\bet_1{\Evec{\V}{\W_2}}^{\perp}\right) \cdot \Evec{\V}{\W_0},\\
     (q_h^\V, \div\w_h^{\t^\perp})
     &=&|K_1|\div\w_h^{\t^\perp}\big|_{K_1}(\V)+|K_2|\div\w_h^{\t^\perp}\big|_{K_2}(\V)
     = \disp\frac{1}2
    \left(\alp_1{\Evec{\V}{\W_1}}^{\perp}-\bet_1{\Evec{\V}{\W_2}}^{\perp}\right) \cdot{\Evec{\V}{\W_0}}^\perp.
    \end{array}
  \end{equation*}
  It can be written in simpler form:
  \begin{equation}\label{eq:qhdivwh0}
    \arraycolsep=1.4pt\def\arraystretch{2.0}
    \begin{array}{lll}
    (q_h^\V, \div\w_h^{\t}) &=&  \disp\frac{\ell_0\ell_1\sin\theta_1}2\alp_1 +\frac{\ell_0\ell_2\sin\theta_2}2\bet_1= |K_1|\alp_1 +|K_2|\bet_1,\\
  (q_h^\V, \div\w_h^{\t^\perp})&=& \disp\frac{\ell_0\ell_1\cos\theta_1}2\alp_1 -\frac{\ell_0\ell_2\cos\theta_2}2\bet_1,
    \end{array}
  \end{equation}
where $\theta_1,\theta_2$ are angles of $K_1, K_2$ at $\V$, respectively, and
$\ell_j=|\overline{\V\W_j}|, j=0,1,2$ as in Figure \ref{fig:union}-(a).

\begin{figure}[ht]
  \centering
  \subfloat[The union of two adjacent triangles sharing $\V$]{
    \includegraphics[width=0.45\textwidth]{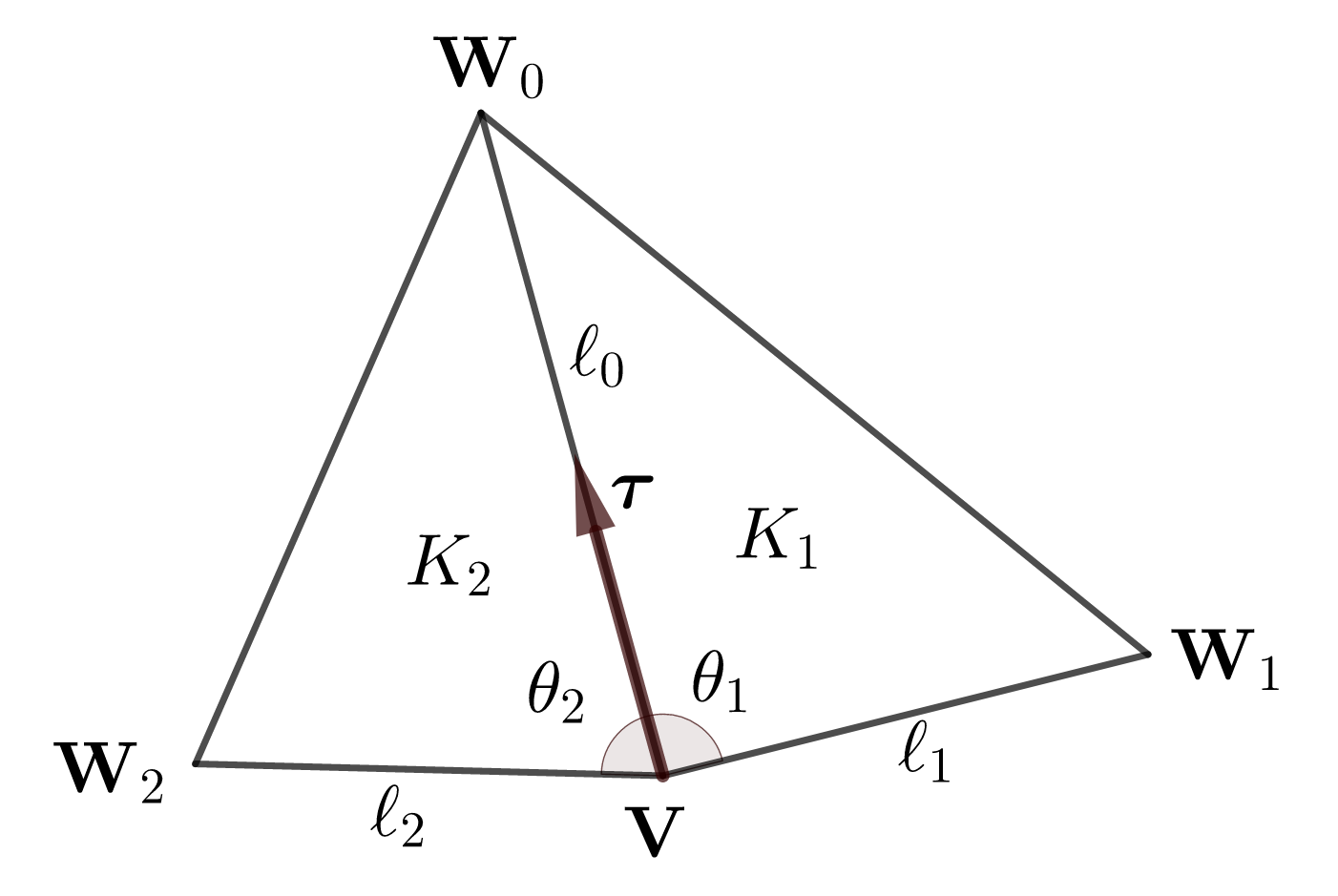} }\qquad
  \subfloat[$\KK(\V)$, the union of all triangles sharing $\V$ ]{
    \includegraphics[width=0.45\textwidth]{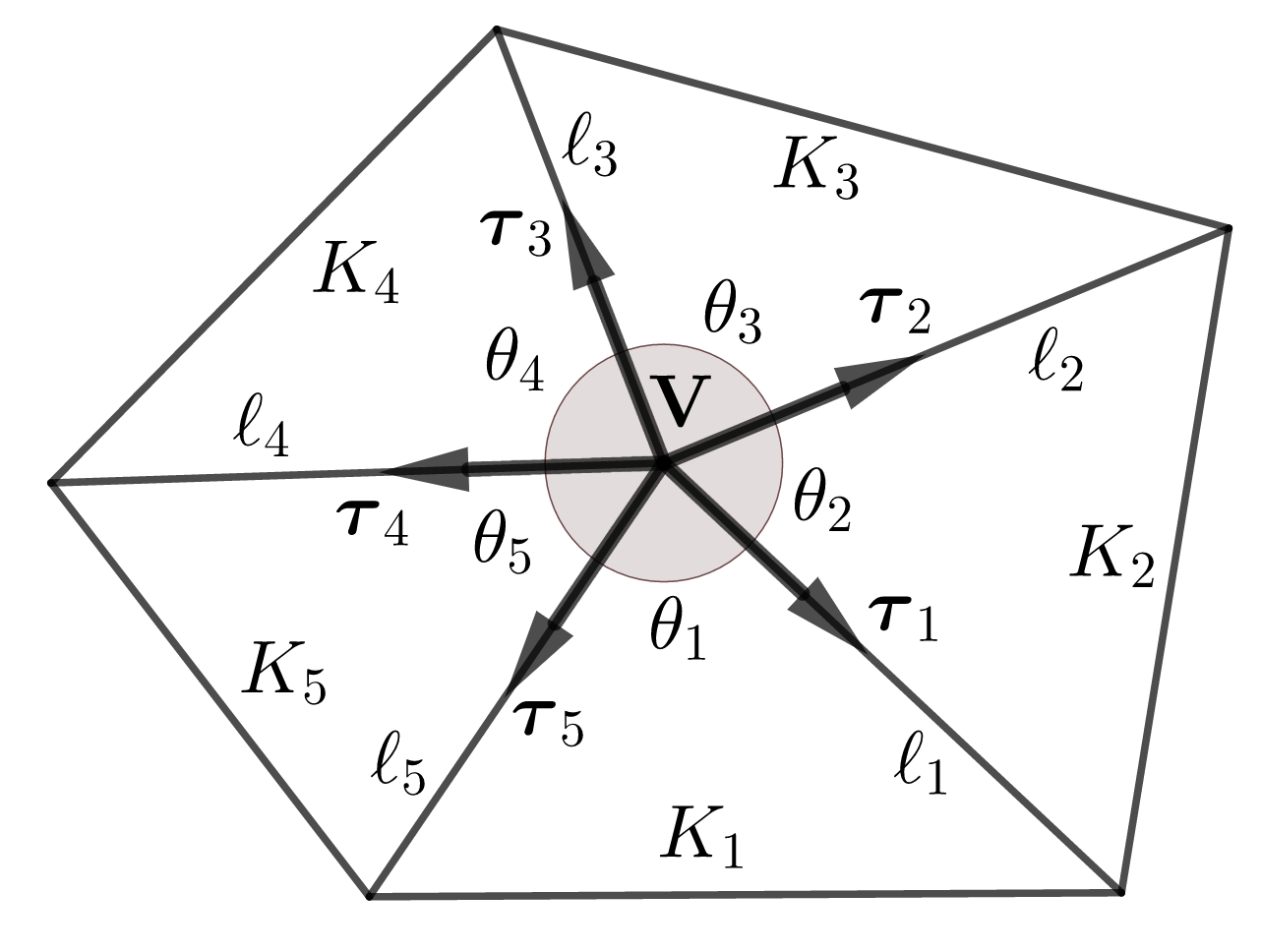}}
  \caption{Unions of triangles}
  \label{fig:union}
\end{figure}

\subsubsection{least square solution}\label{sec:lss}
Let's fix an interior vertex $\V$ and $K_1, K_2, \cdots, K_J$ all triangles in $\Th$
sharing $\V$ as in Figure \ref{fig:union}-(b). We assume the numbering is counterclockwise
and use the index modulo $J$.

For each  $j=1,2,\cdots J$, 
name a vertex $\V_j$ and a unit tangent vector $\t_j$ so that
\begin{equation}\label{not:vktk}
  \overline{\V\V_j}=K_j\cap K_{j+1},\quad \t_j = \frac{\Evec{\V}{\V_j}}{|\overline{\V\V_j}|}.
\end{equation}
There exists a function $w_j\in \mathcal{P}_h^4(\O)\cap H_0^1(\O)$
such that
\begin{equation}\label{cond:wk}
  \frac{\p w_j}{\p\t_j} (\V)=1,\  \frac{\p w_j}{\p\t_j}(\V_j)=0,\
  \int_{\overline{\V\V_j}} w_j\ d\ell=0,\
 \mbox{ the support of  } w_j \mbox{ is } K_j\cup K_{j+1},
\end{equation}
for each  $j=1,2,\cdots J$.
For the uniqueness of $w_j$, we add the following condition:
\begin{equation}\label{cond:wk-redun}
  w_j(\G)=0,\ \nabla w_j(\G)=(0,0) \mbox{ at each gravity center } \G \mbox{ of }
  K_j, K_{j+1}.
\end{equation}
Then we prepare $2J$ test functions in $\Xh$ as
\begin{equation}\label{def:wht}
  \w_h^{\t_j}=w_j\t_j,\quad
  \w_h^{\t_j^\perp}=w_j\t_j^\perp,\quad j=1,2,\cdots,J.
\end{equation}

Let $p_h^{\V}\in \st{h}{}(\V)$, a sum of sting pressures of $\V$,
 be represented as
 \begin{equation}\label{eq:phrepre}
   p_h^\V=\alp_1\st{\V}{K_1}+\alp_2\st{\V}{K_2} +\cdots +
  \alp_J\st{\V}{K_J},
\end{equation}
with $J$ unknown constants $\alp_1, \alp_2,\cdots,\alp_J$.
Consider the following system of $2J$ equations:
\begin{subequations}\label{eq:LSphv}
  \begin{align}
  (p_h^\V, \div\w_h^{\t_j})=(\f,\w_h^{\t_j})-(\nabla\u_h, \nabla\w_h^{\t_j})-(p_h^{\ns{}{}}, \div\w_h^{\t_j}),\label{eq:LSphv-a}\\
    (p_h^\V, \div\w_h^{\t_j^\perp})=(\f,\w_h^{\t_j^\perp})-(\nabla\u_h, \nabla\w_h^{\t_j^\perp})-(p_h^{\ns{}{}}, \div\w_h^{\t_j^\perp}),\label{eq:LSphv-b}
  \end{align}
\end{subequations}
for $j=1,2,\cdots,J$.

\begin{lemma}\label{lem:singval}
  If $\V$ is a regular vertex, then
  the smallest singular value $s$ of the system in \eqref{eq:LSphv} satisfies
  \begin{equation}
    C_\sigma (|K_1|+ |K_2|+\cdots+|K_J|) \le s
  \end{equation}
\end{lemma}
\begin{proof}
  Let
  $\theta_j$ be an angle of $K_j$ at $\V$ and $\ell_j=|\overline{\V\V_j}|, j=1,2,\cdots,J$ as in Figure \ref{fig:union}-(b). Then
  from \eqref{eq:qhdivwh0}, \eqref{cond:wk}, \eqref{eq:phrepre}, we can rewrite \eqref{eq:LSphv} as,
  for $j=1,2,\cdots,J$,
\begin{subequations}\label{eq:srsys0}
   \begin{align}
   \disp\frac{\ell_{j-1}\ell_j\sin\theta_j}2\alp_j &+ \frac{\ell_{j}\ell_{j+1}\sin\theta_{j+1}}2\alp_{j+1}=(\f,\w_h^{\t_j})-(\nabla\u_h, \nabla\w_h^{\t_j})-(p_h^{\ns{}{}}, \div\w_h^{\t_j}),\label{eq:srsys0-a}\\
      \disp\frac{\ell_{j-1}\ell_j\cos\theta_j}2\alp_j &- \frac{\ell_{j}\ell_{j+1}\cos\theta_{j+1}}2\alp_{j+1}
                                           =(\f,\w_h^{\t_j^\perp})-(\nabla\u_h, \nabla\w_h^{\t_j^\perp})-(p_h^{\ns{}{}}, \div\w_h^{\t_j^\perp})\label{eq:srsys0-b}.                                   
   \end{align}
 \end{subequations}

 Denote $\bet_j=\ell_{j-1}\ell_j\alp_j/2$ and
 the right hand side of \eqref{eq:srsys0-a}, \eqref{eq:srsys0-b}
 by $a_j, b_j$ $ j=1,2,\cdots,J$, respectively.
 Then we rewrite \eqref{eq:srsys0} into
   \begin{subequations}\label{eq:srsys1}
 \begin{align}
   \sin\theta_j\bet_j &+ \sin\theta_{j+1}\bet_{j+1}= a_j ,\quad j=1,2,\cdots, J,\label{eq:srsys1-a}\\
  \cos\theta_j\bet_j &- \cos\theta_{j+1}\bet_{j+1}=b_j,\quad j=1,2,\cdots, J.\label{eq:srsys1-b}   
   \end{align}
 \end{subequations}
 If $\V$ is a regular vertex, we can assume without loss of generality,
 \begin{equation}\label{eq:sinth12}
   C_\sigma \le |\sin(\theta_1+\theta_2)|,
 \end{equation}
 which tells the bound of the determinant of two equations in \eqref{eq:srsys1} for $j=1$.

Let $A\in \R^{J\times J}$ be the matrix of the following subsystem of $J$ equations:
 \begin{subequations}\label{eq:srsys2}
 \begin{align}
   \sin\theta_j\bet_j &+ \sin\theta_{j+1}\bet_{j+1}= a_j ,\quad j=1,2,\cdots,J-1,\label{eq:srsys2-a}\\
  \cos\theta_1\bet_1 &- \cos\theta_{2}\bet_{2}=b_1.\label{eq:srsys2-b}   
   \end{align}
 \end{subequations}
 Then from \eqref{eq:sinth12} and $ C_\sigma \le \sin\theta_j, j=1,2,\cdots,J$, we deduce 
 \[  \vertiii{A^{-1}}_2 \le C_\sigma.\]
 It completes the proof since
 the smallest singular value of $A$ is the reciprocal of  $\vertiii{A^{-1}}_2$ and
 the smallest singular value of \eqref{eq:srsys1} is greater than that of its subsystem \eqref{eq:srsys2}.
 \end{proof}
We note that the sting component $(\Pi_hp)^\V\in \st{h}{}(\V)$ of $\Pi_h p$ 
in \eqref{eq:decom-pihp} and $w_j$ in \eqref{cond:wk} satisfies 
\begin{equation}\label{eq:pihpst}
  (\php, w_{j,x})=\left((\Pi_hp)^{\ns{}{}}, w_{j,x}\right)
    + \left((\Pi_hp)^\V, w_{j,x}\right),\quad j=1,2,\cdots,J,
\end{equation}
since $(1,w_{j,x})_{K_j}=(1,w_{j,x})_{K_{j+1}}=0$ and $w_{j,x}$ vanishes at all vertices
except $\V$.

Let $\ph^\V$ be the least square solution of the system \eqref{eq:LSphv}. Then
 the error $(\Pi_hp)^\V-p_h^\V$ is estimated in the following lemma.
\begin{lemma}\label{lem:error-sr0}
  Let $\V$ be a regular vertex. Then,
  the least square solution $p_h^\V$ of  \eqref{eq:LSphv} satisfies 
  \begin{equation}\label{eq:lem-error-phv}
    \|(\Pi_hp)^\V-p_h^\V\|_{0,\KK(\V)} \le C_\sigma
    \left(|\u-\u_h|_{1,\KK(\V)} 
      +\|p-\Pi_hp\|_{0,\KK(\V)}\right).
  \end{equation}
\end{lemma}
\begin{proof}
  From \eqref{eq:pihp0}, \eqref{eq:pihpst} and the definition of  $\w_h^{\t_j}, \w_h^{\t_j^\perp}$
  in \eqref{cond:wk}, \eqref{def:wht}, we have
 \begin{equation}\label{eq:Piphvtk}
 \resizebox{.93\hsize}{!}
  {$
    \arraycolsep=1.2pt\def\arraystretch{1.8}
    \begin{array}{l}
   \left((\Pi_hp)^\V, \div\w_h^{\t_j}\right)=(\f,\w_h^{\t_j})-(\nabla\u, \nabla\w_h^{\t_j})
   -\left((\Pi_hp)^{\ns{}{}}, \div\w_h^{\t_j}\right)-(p-\Pi_hp,\div\w_h^{\t_j}),\\
  \left((\Pi_hp)^\V, \div\w_h^{\t_j^\perp}\right)
  =(\f,\w_h^{\t_j^\perp})-(\nabla\u, \nabla\w_h^{\t_j^\perp})
   -\left((\Pi_hp)^{\ns{}{}}, \div\w_h^{\t_j^\perp}\right)-(p-\Pi_hp,\div\w_h^{\t_j^\perp}),
    \end{array}
    $}
 \end{equation}
 for $j=1,2,\cdots, J$.
 
If we denote the error by $e_h^\V=(\Pi_hp)^\V - p_h^\V$,
then from \eqref{eq:LSphv},\eqref{eq:Piphvtk}, $e_h^\V$ is the least square solution of the following system of $2J$ equations:
\begin{equation}\label{eq:ehvtk}
  \resizebox{.93\hsize}{!}
  {$
  \arraycolsep=1.4pt\def\arraystretch{1.6}
  \begin{array}{l}
   (e_h^\V, \div\w_h^{\t_j})=-\left(\nabla(\u-\u_h), \nabla\w_h^{\t_j}\right)
   -\left((\Pi_hp)^{\ns{}{}}-p_h^{\ns{}{}}, \div\w_h^{\t_j}\right)-(p-\Pi_hp,\div\w_h^{\t_j}),\\
   (e_h^\V, \div\w_h^{\t_j^\perp})=-\left(\nabla(\u-\u_h), \nabla\w_h^{\t_j^\perp}\right)
   -\left((\Pi_hp)^{\ns{}{}}-p_h^{\ns{}{}}, \div\w_h^{\t_j^\perp}\right)
   -(p-\Pi_hp,\div\w_h^{\t_j^\perp}),
  \end{array}
$}
\end{equation} 
 for $j=1,2,\cdots, J$, since the least square solution  of a system is that of its normal equation.
 Now, \eqref{eq:lem-error-phv} comes from  \eqref{est:stVK}, \eqref{eq:ehvtk},
 Lemma \ref{lem:ns-error}, \ref{lem:singval} and 
 $$|w_j|_1\le C_\sigma(|K_j|+|K_{j+1}|)^{1/2},\quad j=1,2,\cdots,J.$$
\end{proof}
Even in case of  regular vertices on $\p\O$, we can copy the above argument to establish
Lemma   \ref{lem:singval} and \ref{lem:error-sr0} with a system of $2(J-1)$ equations.

Denote by $\dotone{\ph}$, the superposition of all the calculated pressures up to now,
that is,
\begin{equation}\label{def:dotone}
\dotone{\ph} = \ph^{\ns{}{}} + \sum_{\V\in\RR_h} \ph^\V,
\end{equation}
where $\RR_h$ is a set of all regular vertices.

\subsection{Step 3. sting pressure for nearly singular vertex except corners}\label{sec:sting-sing}
When a vertex $\V$ is exactly singular, the system \eqref{eq:LSphv} is underdetermined, since
the determinant of \eqref{eq:srsys1} makes
\[ -\sin(\theta_j +\theta_{j+1}) =0\quad \mbox{ for all } j=1,2,\cdots,J.\]
Although $\V$ is not exactly singular, the error $e_h^\V$ in \eqref{eq:ehvtk}
goes through a tiny smallest singular value, if it is nearly singular.

To overcome the problem on nearly singular vertices,
we will replace equations in \eqref{eq:LSphv-b}
with new equations
utilizing jumps of $\dotone\ph$  in \eqref{def:dotone} we have already calculated.  

\subsubsection{boundary singular vertex not a corner}
Let's fix a boundary vertex $\V$
which is not a corner point of $\p\O$. If $\V$ is nearly singular,
it is exact and has only two triangle $K_1, K_2$ which share $\V$ as in Figure \ref{fig:sing}-(a).  Denote by $\V_0,\V_1,\V_2, \t_1$, other 3 vertices and a unit
tangent vector such that
\[ \overline{\V\V_1}= K_1\cap K_2,\quad \V_2\in K_2\setminus\{\V,\V_1\}, \quad
  \V_0\in K_1\setminus\{\V,\V_1\}, \quad
 \t_1 = \frac{\Evec{\V}{\V_1}}{|\overline{\V\V_1}|}. \]

Define the jump of a function $q_h$  across $K_1\cap K_2$ at $\V$ as
\begin{equation}\label{def:jump12}
  \jump_{12}(q_h)
  ={|K_1\cap K_2|^3}\left(\frac{\p}{\p \t_1} \left(q_h\big|_{K_1}\right) (\V) - \frac{\p}{\p \t_1}\left(q_h\big|_{K_2}\right) (\V)\right).
\end{equation}
If $\V$ is nearly singular, then instead of \eqref{eq:LSphv-b}, consider the following new system of $2$ equations:
  \begin{subequations}\label{eq:SSbdphv} 
  \begin{align}
  (p_h^\V, \div\w_h^{\t_1})&=(\f,\w_h^{\t_1})-(\nabla\u_h, \nabla\w_h^{\t_1})-(p_h^{\ns{}{}}, \div\w_h^{\t_1}), \label{eq:SSbdphv-a}\\
   \jump_{12} (p_h^\V) &= -\jump_{12}(\dotone{p_h}).\label{eq:SSbdphv-b}
  \end{align}
\end{subequations}
\begin{lemma}\label{lem:phv-error-svb}
  Let $\V$ be a boundary nearly singular vertex, not a corner of $\p\O$.
  If $p_h^\V$ is the solution of \eqref{eq:SSbdphv}, then we estimate
\begin{equation}\label{eq:lem-error-phvs-svb}
    \|(\Pi_hp)^\V-p_h^\V\|_{0,K_1\cup K_2} \le C_\sigma
    \left(|\u-\u_h|_{1,K_1\cup K_2} +\|p-\Pi_hp\|_{0,K_1\cup K_2}\right).
  \end{equation}
\end{lemma}
\begin{proof}
  Since $\php$ is continuous on $K_1\cap K_2$, we have $\jump_{12}(\php)=0$. It is written in
\begin{equation}\label{eq:jump-php}
  \jump_{12}\left((\php)^\V\right)=-\jump_{12}\left(\php-(\php)^\V\right).
\end{equation}
Thus, if we denote the error by $e_h^\V=(\Pi_hp)^\V - p_h^\V$,
by same argument in regular case, we get
\begin{equation}\label{eq:ehvtk-sbd}
  \resizebox{.93\hsize}{!}
  {$
  \arraycolsep=1.4pt\def\arraystretch{1.6}
  \begin{array}{rll}
   (e_h^\V, \div\w_h^{\t_1})&=&-\left(\nabla(\u-\u_h), \nabla\w_h^{\t_1}\right)
   -\left((\Pi_hp)^{\ns{}{}}-p_h^{\ns{}{}}, \div\w_h^{\t_1}\right)-(p-\Pi_hp,\div\w_h^{\t_1}),\\
    \jump_{12}(e_h^\V)&=&-\jump_{12}\left(\php-(\php)^\V-\dotone{p_h}\right).
  \end{array}
$}
\end{equation} 

We can represent $e_h^\V$ with $2$ unknown constants $e_1, e_2$ as 
\begin{equation}\label{def:ehv-bs}
e_h^\V=e_1 \st{\V}{K_1} + e_2\st{\V}{K_2}.
\end{equation}
Let $\theta_1,\theta_2$ be angles of $K_1, K_2$ at $\V$, respectively, and
$\ell_j=|\overline{\V\V_j}|, j=0,1,2$ as in Figure \ref{fig:sing}-(a).
Then, from the property of sting functions in \eqref{def:sting},
we have
\begin{equation}\label{eq:J12-repre}
    J_{12} (e_h^\V)= -600{\ell_1}^2e_1 +600{\ell_1}^2e_2.
  \end{equation}
  If the first right hand side in \eqref{eq:ehvtk-sbd} is abbreviated by $a$,
  we can rewrite \eqref{eq:ehvtk-sbd} into
\begin{equation}\label{eq:ehvtk-sbd-2}
 \begin{array}{ccccl}
   \disp\frac{\ell_0\ell_1\sin\theta_1}2e_1 &+&\disp
                                                \frac{\ell_{1}\ell_{2}\sin\theta_{2}}2e_{2} &=&a, \\
 -600{\ell_1}^2e_1 &+&600{\ell_1}^2e_2 &=&-\jump_{12}\left(\php-(\php)^\V-\dotone{p_h}\right).
 \end{array}
\end{equation}
We can repeat the same in regular case for the estimation:
\begin{equation}\label{est:norm-a} |a|\le  C_\sigma \left(|K_1|+|K_2|\right)^{1/2}
  \left(|\u-\u_h|_{1,K_1\cup K_2} +\|p-\Pi_hp\|_{0,K_1\cup K_2}\right).
\end{equation}
If we have the following similar estimation for the second right hand side in \eqref{eq:ehvtk-sbd-2}:
\begin{equation}\label{est:norm-jump}
\resizebox{0.93\hsize}{!}
{$
 |\jump_{12}\left(\php-(\php)^\V-\dotone{p_h}\right) |\le  C_\sigma \left(|K_1|+|K_2|\right)^{1/2}
  \left(|\u-\u_h|_{1,K_1\cup K_2} +\|p-\Pi_hp\|_{0,K_1\cup K_2}\right),
$}
\end{equation}
the estimation \eqref{eq:lem-error-phvs-svb} will be established
from \eqref{est:stVK}, \eqref{def:ehv-bs}-\eqref{est:norm-jump}.

If $\V_1$ is an interior vertex, it is regular by Lemma \ref{lem:isol-vtx}.
Even in case of a boundary vertex $\V_1$, it is regular unless $\O=K_1\cup K_2$.
We exclude that pathological case by Assumption \ref{asm:Th}.
Thus, we conclude that $\V_1$ is a regular vertex.

We also note that $\st{\V_0}{K_1}$, $\st{\V_2}{K_2}$ are constant
on $K_1\cap K_2$ from \eqref{def:sting}. It means that they do not have any effect on
 $\jump_{12}(\php)$ and $\jump_{12}(\dotone\ph)$.
 It results in 
\begin{equation}\label{eq:jump-php-svb}
  \jump_{12}\left(\php-(\php)^\V\right) =\jump_{12}\left((\php)^{\ns{}{}}+(\php)^{\V_1}\right).
\end{equation}
Similarly, whether $\V_0, \V_2$ are regular or not, we have
\begin{equation}\label{eq:jump-ph-svb}
\jump_{12}(\dotone{\ph})= \jump_{12}(\ph^{\ns{}{}}+p_h^{\V_1}). 
\end{equation}
From  \eqref{eq:jump-php-svb}, \eqref{eq:jump-ph-svb}, we write
\begin{equation}\label{eq:jumpeh-svb}
  \jump_{12}\left(\php-(\php)^\V-\dotone{p_h}\right)
  =\jump_{12}\left((\php)^{\ns{}{}}- \ph^{\ns{}{}}\right)
  +\jump_{12}\left( (\php)^{\V_1}-\ph^{\V_1}\right).
\end{equation}

For each function $q_h\in\Mh$,  we can estimate
\begin{equation}\label{est:jumpqh}
    \arraycolsep=1.5pt\def\arraystretch{2.0}
  \begin{array}{lll}
   \left| \jump_{12}\left( q_h\right)\right| &\le&
{\ell_1}^3  \left( \left|\left.\nabla q_h \right|_{K_1}(\V)\right|
           +  \left|\left.\nabla q_h \right|_{K_2}(\V)\right| \right)  
                                                           \le C_\sigma{\ell_1}^3 \left(|K_1|+|K_2|\right)^{-1/2}\left|q_h\right|_{1,K_1\cup K_2}\\
                                             &\le& C_\sigma{\ell_1}^2 \left|q_h\right|_{1,K_1\cup K_2}\le C_\sigma{\ell_1} \left\|q_h\right\|_{0,K_1\cup K_2}  \le C_\sigma \left(|K_1|+|K_2|\right)^{1/2}\left\|q_h\right\|_{0,K_1\cup K_2}.
  \end{array}                                                        
\end{equation}
Thus we obtain \eqref{est:norm-jump} from \eqref{eq:jumpeh-svb}, \eqref{est:jumpqh} and
Lemma \ref{lem:ns-error}, \ref{lem:error-sr0}.
\end{proof}
\begin{figure}[ht]
  \hspace{10mm}
  \subfloat[$\V\in\p\O\setminus\{\mbox{corners of }\p\O\}$]{
    \includegraphics[width=0.38\textwidth]{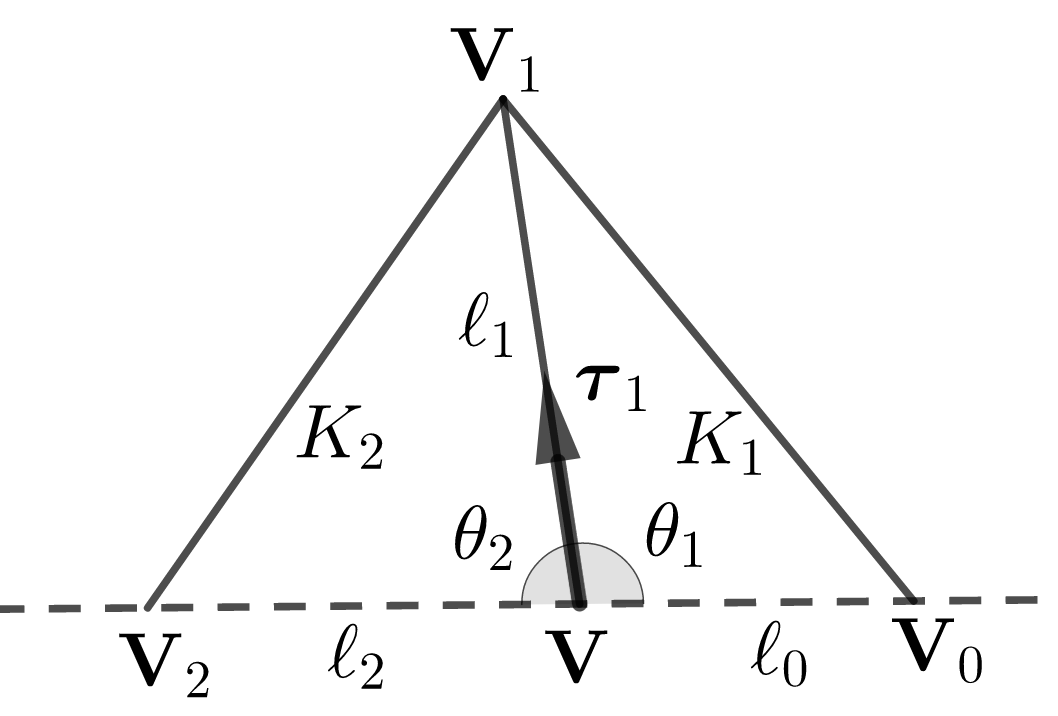} }\qquad
  \subfloat[ a corner $\V$ of $\p\O$.]{
    \includegraphics[width=0.38\textwidth]{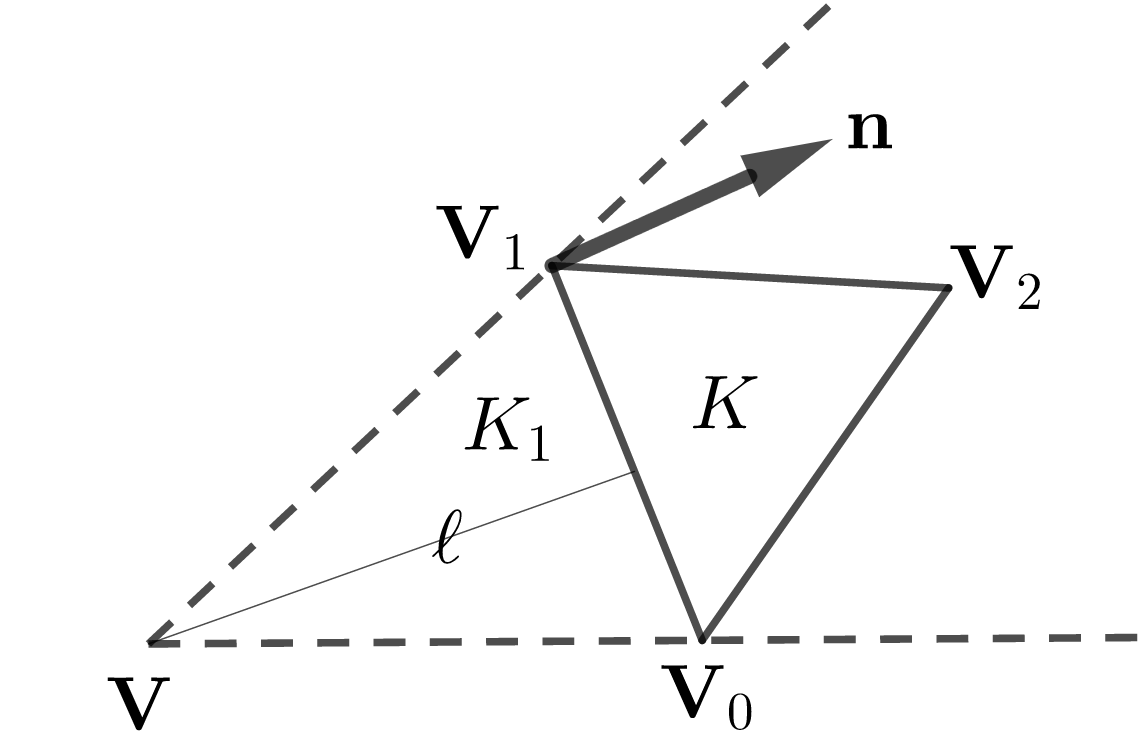}}
  \caption{examples of boundary singular vertices (dashed lines belong to $\p\O$) }
  \label{fig:sing}
\end{figure}

\subsubsection{interior singular vertex}
Let $\V$ be an interior vertex and go back to the setting in Section \ref{sec:lss}
as in Figure \ref{fig:union}-(b).
For each $j=1,2,\cdots,J$, define the jump of a function $q_h$  across $K_j\cap K_{j+1}$ at $\V$ as
\begin{equation}\label{def:jumpkk1}
  \jump_{\jjp}(q_h)
  ={|K_j\cap K_{j+1}|^3}\left(\frac{\p}{\p \t_j} \left(q_h\big|_{K_j}\right) (\V) - \frac{\p}{\p \t_j}\left(q_h\big|_{K_{j+1}}\right) (\V)\right).
\end{equation}
If $\V$ is nearly singular,
replace \eqref{eq:LSphv-b} with new equations using the jumps of $\dotone\ph$ in \eqref{def:jumpkk1}. That is, we consider the following system of $2J$ equations:
  \begin{subequations}\label{eq:SSphv-svi} 
  \begin{align}
  (p_h^\V, \div\w_h^{\t_j})&=(\f,\w_h^{\t_j})-(\nabla\u_h, \nabla\w_h^{\t_j})-(p_h^{\ns{}{}}, \div\w_h^{\t_j}), \label{eq:SSphv-svi-a}\\
    \jump_{\jjp} (p_h^\V) &= -\jump_{\jjp}(\dotone{p_h}),\quad
                            j=1,2,\cdots,J.\label{eq:SSphv-ssi-b}
  \end{align}
\end{subequations}

Let $\ph^\V$ be the least square solution of \eqref{eq:SSphv-svi}.
We note that  all other adjacent vertices $\V_1, \V_2,\cdots,\V_J$
are regular from Lemma \ref{lem:isol-vtx}.
Thus we can repeat the arguments
in the proofs of Lemma \ref{lem:singval}, \ref{lem:error-sr0}, \ref{lem:phv-error-svb}
to establish the following lemma.
\begin{lemma}\label{lem:phv-error-svi}
  Let $\V$ be an interior nearly singular vertex.
  If $p_h^\V$ be the least square solution of \eqref{eq:SSphv-svi}, then we estimate
\begin{equation}\label{eq:lem-error-phvs}
    \|(\Pi_hp)^\V-p_h^\V\|_{0,\KK(\V)} \le C_\sigma
    \left(|\u-\u_h|_{1,\KK(\V)} +\|p-\Pi_hp\|_{0,\KK(\V)}\right).
  \end{equation}
\end{lemma}

Denote by $\dottwo{\ph}$, the superposition of all the calculated pressures up to now,
that is,
\begin{equation}
\dottwo{\ph} = \ph^{\ns{}{}} + \sum_{\V\in\VV_h'} \ph^\V,
\end{equation}
where $\VV_h'$ is a set of all vertices except nearly singular corner points of $\p\O$.

\subsection{Step 4.  sting pressure for nearly singular  corner}\label{sec:sting-corner}
Let $\V$ be a corner point of $\p\O$ which is nearly singular
and $K_1, K_2,\cdots, K_J$, triangles in $\Th$ sharing $\V$.
If $J\ge 2$, 
we can calculate the least square solution $\ph^\V$
satisfying the system of $2(J-1)$ equations as in \eqref{eq:SSphv-svi}
and estimate its error $(\php)^\V-\ph^\V$ as in \eqref{eq:lem-error-phvs}.

If $J=1$, there exists a triangle $K$ in $\Th$ sharing 
two vertices $\V_0, \V_1$ with $K_1$ as in Figure \ref{fig:sing}-(b). Denote by $\V_2$,
the third vertex of $K$ not shared with $K_1$.

Define the jump of a function $q_h$  across $K_1\cap K$ at $\V_1$ as
\begin{equation}\label{def:jump12n}
  \jump(q_h)
  ={\ell^3}\left(\frac{\p}{\p \n} \left(q_h\big|_{K_1}\right) (\V_1)
    - \frac{\p}{\p \n}\left(q_h\big|_{K}\right) (\V_1)\right),
\end{equation}
where $\n$ is a unit outward normal vector on $K_1\cap K$ of $K_1$ and $\ell$ is the distance
between $\V$ and $K_1\cap K$.
Then we can determine $\ph^\V=\alp\st{\V}{K_1}$ for some constant $\alp$
satisfying
\begin{equation}\label{eq:lastphv}
  \jump(\ph^\V)=-\jump(\dottwo \ph).
\end{equation}
\begin{lemma}\label{lem:phv-error-corner}
  Let $\V$ be a corner which meets only one triangle.
  If $p_h^\V$ is the solution of \eqref{eq:lastphv}, then we estimate
\begin{equation}\label{eq:lem:phv-error-corner-0}
    \|(\Pi_hp)^\V-p_h^\V\|_{0,K_1} \le C_\sigma
    \left(|\u-\u_h|_{1,K_1\cup K} +\|p-\Pi_hp\|_{0,K_1\cup K}\right).
  \end{equation}
\end{lemma}
\begin{proof}
  We note that $\nabla\php$ is continuous at $\V_1$, since it is a Hermite interpolation of $p$
  in \eqref{def:php}.
It can be written in
\begin{equation}\label{eq:jump-php-w}
  \jump\left( (\php)^\V   \right) =-\jump\left( \php -  (\php)^\V \right).
\end{equation}
Set
the error $e_h^\V=(\php)^\V-\ph^\V=e\st{\V}{K_1}$ for some constant $e$,
then from \eqref{eq:lastphv}, \eqref{eq:jump-php-w}, we have
\begin{equation}
  \label{eq:jump-last-phpph}
  \jump(e\st{\V}{K_1}) = -\jump\left(\php -  (\php)^\V-\dottwo\ph\right).
\end{equation}
From the property of sting functions in \eqref{def:sting}, we have
\begin{equation}
  \label{eq:lem:phv-corner-a}
   \jump(e\st{\V}{K_1})=-180{\ell}^2e.
 \end{equation}
 We have already calculated $\ph^{\V_0}, \ph^{\V_1} $, $\ph^{\V_2}$,
 since $\V_0, \V_1, \V_2$ are not corners  by Assumption \ref{asm:Th}.
Thus the difference  in  \eqref{eq:jump-last-phpph} is
written as
\begin{equation}
  \label{eq::lem:phv-corner-b}
  \php -  (\php)^\V-\dottwo \ph= (\php)^{\ns{}{}}-\ph^{\ns{}{}} +
\sum_{j=0}^2(\php)^{\V_j}-\ph^{\V_j}  + (\php)^C.
\end{equation}

For each function $q_h\in\Mh$,  we can estimate the following as in \eqref{est:jumpqh},
\begin{equation}\label{est:jumpqh2}
 \left| \jump\left( q_h\right)\right| \le C_\sigma \left(|K_1|+|K|\right)^{1/2}\left\|q_h\right\|_{0,K_1\cup K}.
\end{equation}
Now we can reach at \eqref{eq:lem:phv-error-corner-0}
with \eqref{est:stVK}, \eqref{eq:jump-last-phpph}-\eqref{est:jumpqh2} and
Lemma \ref{lem:ns-error}, \ref{lem:error-sr0}-\ref{lem:phv-error-svi}
since $\jump\left( (\php)^C\right)=0$.
\end{proof}

Denote by $\dotthr{\ph}$, the superposition of all the calculated pressures up to now,
\begin{equation}\label{def:dot3p}
\dotthr{\ph} = \ph^{\ns{}{}} + \sum_{\V\in\VV_h} \ph^\V,
\end{equation}
where $\VV_h$ is a set of all vertices in $\Th$. For notation consistency,
denote again
\begin{equation}\label{def:dot3php}
  \dotthr{\php}=\php-(\php)^{\CC}=(\php)^{\ns{}{}} + \sum_{\V\in\VV_h} (\php)^\V.
\end{equation}
Then hereby, we have established the following lemma. 
\begin{lemma}\label{lem:prs-error-0}
\begin{equation*}
    \left\|\dotthr\php-\dotthr\ph\right\|_0
    \le C_\sigma (|\u-\u_h|_1+\|p-\php\|_0).
  \end{equation*}
\end{lemma}
\subsection{Step 5. piecewise constant pressure}\label{sec:con-prs}
Let $Y_h$ be a space of all functions in $\left[\mathcal{P}_h^2(\O)\cap H_0^1(\O)\right]^2$
vanishing at all vertices and whose value at the midpoint of each edge $E$
is normal to $E$.
Define
\[X_h=\left[\mathcal{P}_h^1(\O)\cap H_0^1(\O)\right]^2 \bigoplus Y_h.\]
 Then, there exists a unique
$(\w_h, \ph^{c})\in X_h\times\mathcal{P}_h^0(\O)\cap L_0^2(\O)$ which satisfies
a following discrete Stokes problem:
\begin{equation}\label{eq:Stokes2}
  \begin{array}{rll}
  (\nabla\w_h, \nabla\v_h) + (\ph^{c},\div\v_h)
    &=& (\f,\v_h)-(\nabla\u_h,\nabla\v_h)-(\dotthr\ph,\div\v_h),\\
    (q_h,\div\w_h)&=&0,
  \end{array}
\end{equation}
for all $(\v_h,q_h)\in  X_h\times\mathcal{P}_h^0(\O)\cap L_0^2(\O)$,
since $X_h\times \mathcal{P}_h^0(\O)\cap L_0^2(\O)$ satisfies the inf-sup condition
 \cite{Bernardi1985, Brezzi-Fortin}.

Denote by $\m(f)$, the average of a function $f$ over $\O$, then define
\begin{equation}\label{def:phc}
  \ph^{\CC}=\ph^c -\m(\dotthr\ph).
\end{equation}
\begin{lemma}\label{lem:prs-error-1}
  \begin{equation}\label{eq:th:prs-error-1}
    \|(\php)^{\CC}-\ph^{\CC}\|_0 \le  C_\sigma (|\u-\u_h|_1+\|p-\php\|_0).
  \end{equation}
\end{lemma}
\begin{proof}
 From \eqref{eq:pihp0}, \eqref{def:dot3php}, we have for all $\v_h\in X_h$,
  \begin{equation}
    \label{eq:Stokes0-a}
    \left((\php)^{\CC},\div\v_h\right)=(\f,\v_h)-(\nabla\u,\nabla\v_h)-(p-\Pi_hp,\div\v_h)
    -\left(\dotthr{\php},\div\v_h\right).
  \end{equation}
  Let's decompose $X_h$ as 
  \begin{equation}\label{eq:decom-Xh}
    \arraycolsep=1.4pt\def\arraystretch{1.6}
    \begin{array}{lll}
      W_h&=&\{\x_h\in X_h\ |\ (q_h,\div\x_h)=0 \mbox{ for all } q_h\in \mathcal{P}_h^0(\O)\cap L_0^2(\O)\},\\
   W_h^{\perp}&=&\{\z_h\in X_h\ |\ (\nabla\x_h,\nabla\z_h)=0\ \mbox{ for all }
    \x_h\in W_h\}. 
    \end{array}
  \end{equation}

  Then, we can rewrite \eqref{eq:Stokes2} for all  $\z_h\in W_h^{\perp}$,
  \begin{equation}\label{eq:Stoeks2-a}
    (\ph^c,\div\z_h)=(\f,\z_h)-(\nabla\u_h,\nabla\z_h)-(\dotthr\ph,\div\z_h). 
  \end{equation}
  Denote the error by $e_h^{\CC}=(\php)^{\CC}-\ph^{c}$.
  Then from \eqref{eq:Stokes0-a}, \eqref{eq:Stoeks2-a}, we have  for all  $\z_h\in W_h^{\perp}$,
  \begin{equation}\label{eq:diff-const}
  \left(e_h^{\CC},\div\z_h \right)=-(\nabla\u-\nabla\u_h,\nabla\z_h)-(p-\Pi_hp,\div\z_h)
    -\left(\dotthr{\php}-\dotthr\ph,\div\z_h\right).
  \end{equation}
 Since $e_h^{\CC}-\m(e_h^{\CC})\in L_0^2(\O)$, there exists a nonzero $\v_h\in X_h$ such that
  \begin{equation}
    \label{eq:inf-sup-p2p0}
    \|e_h^{\CC}-\m(e_h^{\CC})\|_0\ |\v_h|_1 \le \beta (e_h^{\CC}-\m(e_h^{\CC}),\div\v_h),
  \end{equation}
from the inf-sup condition, where $0<\beta$ is a constant, independent of $\Th$.

  If we decompose $\v_h$  orthogonally  into $\x_h\in W_h$ and $\z_h\in W_h^\perp$ such that
  \begin{equation}
    \label{eq:decom-vhxhzh}
    \v_h=\x_h + \z_h,
  \end{equation}
  from \eqref{eq:diff-const}-\eqref{eq:decom-vhxhzh} and Lemma \ref{lem:prs-error-0}, we expand
  \begin{multline} \label{eq:last-expan}
    \|e_h^{\CC}-\m(e_h^{\CC})\|_0\ |\z_h|_1 \le  \|e_h^{\CC}-\m(e_h^{\CC})\|_0\ |\v_h|_1\le \beta (e_h^{\CC}-\m(e_h^{\CC}),\div\v_h)\\
    =\beta (e_h^{\CC}-\m(e_h^{\CC}),\div\z_h)=\beta (e_h^{\CC},\div\z_h)
    \le  \beta C_\sigma(|\u-\u_h|_1+\|p-\php\|_0)|\z_h|_1.
  \end{multline}
  If $z_h=\mathbf{0}$, then $e_h^{\CC}-\m(e_h^{\CC})=0$ from  \eqref{eq:inf-sup-p2p0}, \eqref{eq:decom-vhxhzh}.  Therefore, from \eqref{eq:last-expan}, we have
  \begin{equation}
    \label{est:eh-meh}
     \|e_h^{\CC}-\m(e_h^{\CC})\|_0 \le  C_\sigma(|\u-\u_h|_1+\|p-\php\|_0).
  \end{equation}

  Now, since $\m(e_h^{\CC})=\m\left((\php)^{\CC}\right)$, we expand from
  \eqref{def:dot3p}, \eqref{def:dot3php}, \eqref{def:phc} that
  \begin{equation}
    \label{eq:last-huddle}\arraycolsep=1.4pt\def\arraystretch{1.6}
    \begin{array}{lll}
      (\php)^{\CC}-\ph^{\CC} &=&    (\php)^{\CC}-\left(\ph^{c}-\m(\dotthr\ph)\right)
           =e_h^{\CC}-\m(e_h^{\CC})+\m\left((\php)^{\CC}\right) +\m(\dotthr\ph)\\
                                                     &=& e_h^{\CC}-\m(e_h^{\CC}) + \m(\php)-\m\left(\dotthr\php\right)+\m\left(\dotthr\ph \right).
    \end{array}
  \end{equation}
  For the averages in \eqref{eq:last-huddle}, we observe the following two basic inequalities:
  \begin{equation}\label{est:mean-php}
    |\m(\php)|=|\m(\php-p)|
    \le |\O|^{-1/2}\|\php-p\|_0,
  \end{equation}
  \begin{equation}\label{est:mean-phpph}
    \left|\m\left(\dotthr\php\right)-\m\left(\dotthr\ph \right)\right|=\left|\m\left(\dotthr\php-\dotthr\ph\right)\right|\le  |\O|^{-1/2}\|\dotthr\php-\dotthr\ph\|_0.
  \end{equation}
  Therefore, \eqref{eq:th:prs-error-1} is established from
  \eqref{est:eh-meh}-\eqref{est:mean-phpph} and Lemma \ref{lem:prs-error-0}.
\end{proof}

Now  we reach at Theorem \ref{th:prs-error} with the definitions \eqref{def:ph}, \eqref{def:dot3p}, \eqref{def:dot3php} and  Lemma \ref{lem:prs-error-0}, \ref{lem:prs-error-1} and Theorem \ref{th:vel-error}.

\section{Numerical test}
We tested the suggested successive method in $\O=[0,1]^2$ with the velocity $\u$
and pressure $p$ such that
\[ \u=\left(s(x)s'(y), -s'(x)s(y)\right),\quad
  p=\sin(4\pi x)e^{\pi y},\quad\mbox{ where }s(t)=(t^2-t)\sin(2\pi t).  \]

For triangulations, we first formed the meshes of uniform squares
over $\O$, then added one exactly singular vertex in every square.
An example of $8\times 8\times 4$ mesh is given in Figure \ref{fig:mesh}.

We have calculated the successive pressures
$\ph^{\ns{}{}}, \dotone\ph, \dottwo\ph$ and $\ph^{\CC}$
along to Step 1,2,3,5 in Section \ref{sec:main}.
Since the meshes have not any singular corner, $\dotthr\ph=\dottwo\ph$.
Their examples are depicted in Figure \ref{fig:ph} as well as
$\ph=\dotthr\ph + \ph^{\CC}$.

The error table in Table \ref{table} shows the optimal order of convergence,
expected in Theorem \ref{th:vel-error} and \ref{th:prs-error}.
We adopted a direct linear solver in {\texttt {LAPACK}}
on solving the systems from the problem \eqref{eq:dclv-SC} and \eqref{eq:Stokes2}
to focus on testing the suggested method.

\begin{figure}[ht]
  \centering
  \includegraphics[width=0.35\linewidth]{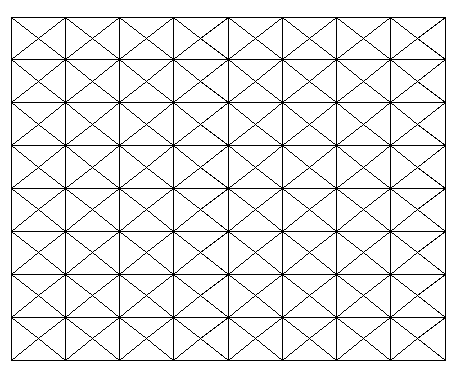}
\caption{ $8\times8\times4$ mesh bearing $8\times8$ interior exactly singular vertices}
\label{fig:mesh}
\end{figure}

\begin{figure}[ht]
 \subfloat[$\ph^{\ns{}{}}$]
 {\includegraphics[width=0.49\textwidth]{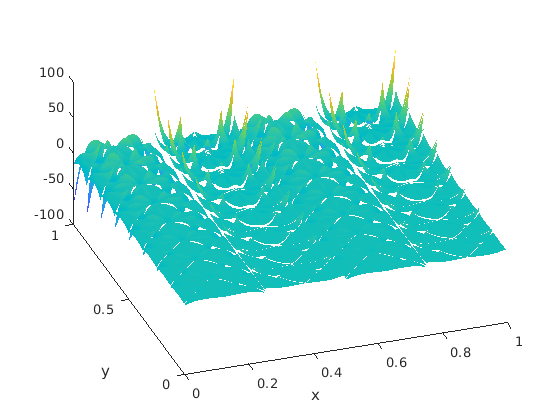}}
  \subfloat[$\dotone\ph=\ph^{\ns{}{}}+\sum \ph^{\V_r}$]
  {\includegraphics[width=0.49\textwidth]{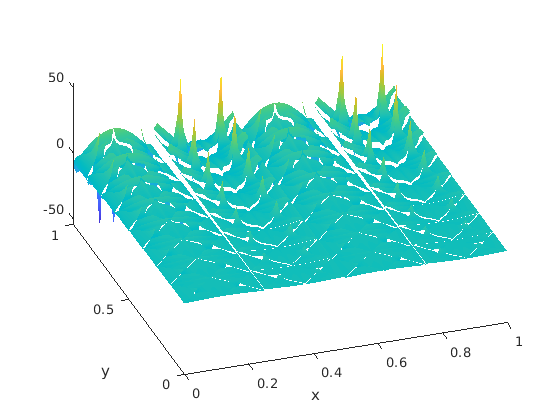}}\\
  \subfloat[$\dottwo\ph=\dotone\ph+\sum \ph^{\V_s},\quad \dotthr\ph=\dottwo\ph$]
  {\includegraphics[width=0.49\textwidth]{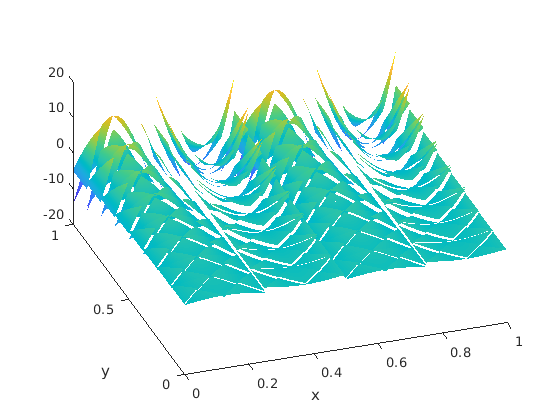}}
  \subfloat[$\ph^{\CC}$ ]
  {\includegraphics[width=0.49\textwidth]{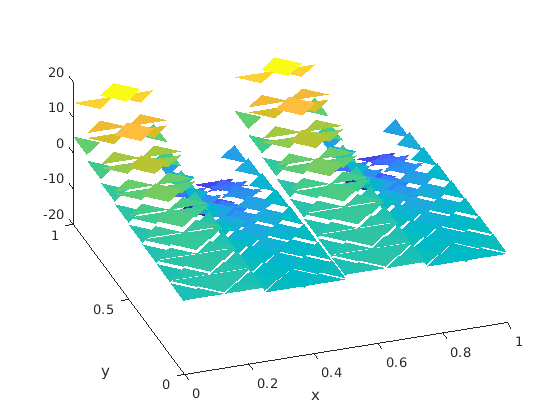}}\\
  \centering \subfloat[$\ph=\dotthr\ph + \ph^{\CC}$]
 {\includegraphics[width=0.49\textwidth]{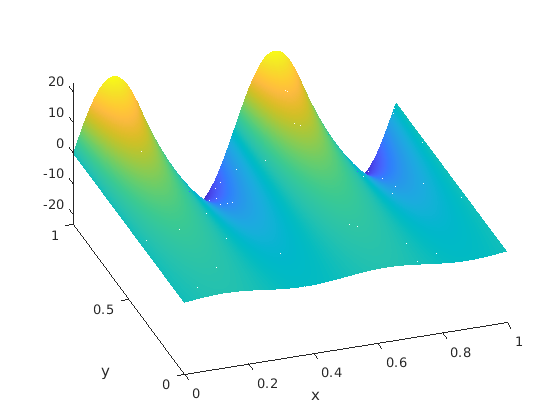}}
 \caption{successive pressures calculated in $8\times8\times4$ mesh in Figure \ref{fig:mesh}}
  \label{fig:ph}
\end{figure}

\begin{table}
  \centering
\begin{tabular}{r || c c || c c }
\hline
  mesh\hspace{3mm}
  & \hspace{1mm}$|\u-\u_h|_1$\hspace{1mm} & order\hspace{1mm}
     & \hspace{1mm}$\|p-p_h\|_0$\hspace{1mm}
  & order \\
\hline
4 x 4 x 4 & 1.1264E-2  & &  5.8000E-2 & \\ 
 8 x 8 x 4 & 6.1498E-4  &4.1950  &2.7012E-3  &  4.4244  \\ 
 16 x 16 x 4 &  3.5942E-5 &4.0968  & 1.6760E-4   &4.0105  \\
 32 x 32 x 4 &2.2002E-6  &4.0299  &  1.0454E-5  &  4.0029 \\ 
\hline
\end{tabular}
\caption{\label{table}error table}
\end{table}

\end{document}